\documentclass[a4paper,oneside,10pt]{amsproc} 

\usepackage[ansinew]{inputenc} 
\usepackage[T1]{fontenc}
\usepackage[english]{babel} 

\usepackage{enumerate} 
\usepackage{booktabs}

\usepackage{amsmath} 
\usepackage{amsfonts} 
\usepackage{amssymb}
\usepackage{amsthm}
\usepackage{latexsym}
\usepackage{bbm}
\usepackage{setspace}
\usepackage[top=3cm , bottom=3cm , left=3cm , right=3cm]{geometry}
\usepackage[pdftex]{graphicx}
  \DeclareGraphicsExtensions{.pdf,.png,.jpg,.jpeg,.mps}
  \usepackage{pgf}
  \usepackage{tikz}
\usepackage{floatflt}
\usepackage{url}
\usepackage{mathrsfs}
\usepackage{esint}

\makeatletter

\@namedef{subjclassname@2010}{%

  \textup{2010} Mathematics Subject Classification}

\makeatother

\theoremstyle{plain}
\newtheorem{prop}{Proposition}
\newtheorem{lemma}[prop]{Lemma}
\newtheorem*{cone*}{Cone covering property} 
\newtheorem{theorem}[prop]{Theorem}
\newtheorem*{prop*}{Proposition}
\newtheorem*{lemma*}{Lemma}
\newtheorem{cor}[prop]{Corollary}
\newtheorem*{cor*}{Corollary}
\newtheorem*{theorem*}{Theorem}

\newtheorem{mainresult}{Main result}

\theoremstyle{definition} 

\newtheorem*{def*}{Definition}

\theoremstyle{remark}

\newtheorem*{remark}{Remark}
\newtheorem*{remarks}{Remarks}
\newtheorem*{acks}{Acknowledgements}

\newtheorem*{example}{Example}

\newcommand{\R}{\mathbb{R}}

\newcommand{\Z}{\mathbb{Z}}
\newcommand{\C}{\mathbb{C}}
\newcommand{\E}{\mathbb{E}}

\newcommand{\D}{\,\textup{d}}
\newcommand{\Dt}{\,\frac{\textup{d} t}{t}}
\newcommand{\Ds}{\,\frac{\textup{d} s}{s}}

\newcommand{\Dd}{\mathscr{D}}

\newcommand{\Ll}{\mathscr{L}}

\newcommand{\Aa}{\mathscr{A}}

\newcommand{\la}{\langle}
\newcommand{\ra}{\rangle}

\newcommand{\supp}{\textup{supp}\,}

\newcommand{\ran}{\textsf{R}}
\newcommand{\dom}{\textsf{D}}

\title[Vector-valued tent spaces and Hardy spaces]{On vector-valued tent spaces and Hardy spaces associated with non-negative self-adjoint operators}
\author{Mikko Kemppainen}
\address{Departamento de Matem\'{a}ticas, Universidad Aut\'{o}noma de Madrid, 28049 Madrid, Spain}
\curraddr{Department of Mathematics and Statistics, University of Helsinki, FI-00014 Helsinki, Finland}
\email{mikko.k.kemppainen@helsinki.fi}
\date{\today}

\begin{document}

\begin{abstract}
  In this paper we study Hardy spaces associated with non-negative self-adjoint operators and develop their vector-valued theory. The complex interpolation scales of vector-valued tent spaces and Hardy spaces are extended to the endpoint $p=1$. The holomorphic functional calculus of $L$ is also shown to be bounded on the associated Hardy space $H^1_L(X)$. These results, along with the atomic decomposition for the aforementioned space, rely on boundedness of certain integral operators on the tent space $T^1(X)$.
\end{abstract}


\maketitle


\section{Introduction}

The theory of Hardy spaces associated with operators has been extensively studied in the recent years. 
Indeed, the cases of elliptic operators on $\R^n$ \cite{HOFMANNDIV,SECONDORDER}, 
non-negative self-adjoint operators on doubling metric measure spaces \cite{HOFMANNHARDY} and
Hodge--Dirac operators on Riemannian manifolds (with doubling volume measure) \cite{AMR}
are all well-understood by now.

In the abovementioned cases, the Hardy spaces are defined in terms of conical square functions, which has the benefit
of allowing a direct connection with \emph{tent spaces}. These were first introduced 
by Coifman, Meyer and Stein in \cite{CMSTENTSPACES} and have
since become a central tool in Harmonic Analysis. Their theory extends without much difficulty to
doubling metric measure spaces (see \cite{AMENTA,RUSS}).

The aim of this paper is to study such Hardy spaces for functions that take their values in an infinite dimensional
Banach space. This is not a completely new development; the theory of vector-valued Hardy spaces associated with 
bisectorial operators on $\R^n$ was initiated by Hyt{\"o}nen, van Neerven and Portal in \cite{HVNPCONICAL}, 
which is the main inspiration for this article.
However, their theory covers only the range $1 < p < \infty$, mainly because not all of the classical scalar-valued 
tent space techniques carry over to vector-valued setting. A new method, suitable for vector-valued tent spaces,
was introduced by the author in \cite{TENTSPACES}, which allowed to extend the theory to $p=1$.
In this article we study the case of vector-valued Hardy spaces associated with non-negative self-adjoint operators
on certain doubling metric measure spaces and develop the corresponding theory of tent spaces.

The main result concerning interpolation (Theorem \ref{interpolationendpoint} and Corollary 
\ref{tentmain1}) extends Theorem 4.7 from \cite{HVNPCONICAL} to the lower endpoint.

\begin{mainresult}
  The complex interpolation scale of vector-valued tent spaces $T^p(X)$ extends to $p=1$.
\end{mainresult}

Actually, also the other endpoint $T^\infty (X)$ is included in the interpolation scale as a consequence of the duality
$T^1(X)^* \simeq T^\infty (X^*)$ (Theorem \ref{endpointduality}, cf. \cite[Theorem 14]{TENTSPACES}). The `classical' proof of the
duality \cite[Theorem 1(b)]{CMSTENTSPACES} becomes available in the vector-valued setting after a more direct
definition of tent spaces which does not rely on completions (see Section \ref{tentspaces} and Appendix A).

Instead of the `embedding method' from \cite{HARBOURE} and \cite{TENTSPACES} (which for $p=1$ and $p=\infty$ 
is of a strictly Euclidean nature),
the proof of Main result 1 is based on
a geometric assumption on the underlying space, namely the \emph{cone covering property}. It is meant as an abstraction
of the proof technique rather than a genuine geometric property, and the framework of metric measure spaces is chosen
primarily to highlight the flexibility of this method. In \cite{TENTSPACES} it was proven for $\R^n$ and in 
\cite{NONNEGATIVE} it is shown to hold, more generally, on
complete (connected) Riemannian manifolds of non-negative sectional curvature.

The communication between tent spaces and Hardy spaces happens by means of integral operators.
In the vector-valued setting the boundedness of integral operators on tent spaces
relies on the change of aperture \cite[Theorems 4.3 and 5.6]{HVNPCONICAL}.
We obtain a change of aperture inequality on $T^1(X)$ from the \emph{atomic decomposition}, the proof of which also relies on 
the cone covering property, 
and extend the integral operators to $T^1(X)$ following closely the proof from \cite{HVNPCONICAL}.

We then arrive at the second main result (Theorems \ref{main2} and \ref{main1}), 
which extends Theorem 7.10 and Corollary 7.2 from \cite{HVNPCONICAL} to the
endpoint $p=1$:

\begin{mainresult}
  The complex interpolation scale of vector-valued Hardy spaces $H^p_L(X)$ extends to $p=1$.
  Moreover, $L$ has a bounded $H^\infty$-functional calculus on $H^1_L(X)$. 
\end{mainresult}

It is well-understood that the tent space atomic decomposition can be turned into
atomic or molecular decomposition of the Hardy space (see Theorem \ref{Hardyatomicdec}):

\begin{mainresult}
  Functions in a dense subspace of $H^1_L(X)$ admit decompositions into atoms.
\end{mainresult}

As a corollary, the `square function Hardy space' $H^1_\Delta (X)$ 
associated with the (non-negative) Laplacian $\Delta$ on $\R^n$ 
coincides with the classical `atomic Hardy space'. The presented framework also covers the case when $L$ is the
Laplace--Beltrami operator on a complete (connected) Riemannian manifold with non-negative sectional curvature.

The vector-valued tent space theory makes use of pointwise estimates, which imposes two limitations to the current understanding. 
Firstly, in order to have atomic decompositions and interpolation for tent spaces we rely
on the cone covering property of the underlying metric space.  
Secondly, for non-self-adjoint operators, it is by no means clear how to obtain molecular decompositions for the associated 
Hardy spaces. The difficulty arises in the attempt to interpret the molecular decay condition by means of
integral operators on tent spaces.

\begin{acks}
  The author gratefully acknowledges financial support from the V\"ais\"al\"a Foundation and 
  from the Academy of Finland through the project 
  \emph{Stochastic and harmonic analysis: interactions and applications} (133264). Many thanks to the anonymous referee for carefully reading
  the manuscript and offering suggestions for improvement.
\end{acks}

\section{Preliminaries}

\subsection*{Notation}

Random variables are taken to be defined on a fixed probability space whose 
expectation is denoted by $\E$.
Given a Banach space $X$ the duality pairing between $\xi\in X$ and $\xi^*\in X^*$ is written as $\la \xi , \xi^* \ra$.
By $\alpha \lesssim_\varepsilon \beta$ it is meant that there exists a constant $C_\varepsilon$ (depending on a parameter $\varepsilon$) such that
$\alpha \leq C_\varepsilon \beta$. Quantities $\alpha$ and $\beta$ are comparable, $\alpha \eqsim \beta$, if
$\alpha \lesssim \beta$ and $\beta \lesssim \alpha$.

\subsection*{Stochastic integration and $\gamma$-radonifying operators}
We first recall some facts about stochastic integration of functions with values in a (complex) Banach space
(see \cite{JVNSTOCHINT} for details).

Let $(\Omega,\nu)$ be a $\sigma$-finite measure space and assume that a \emph{random measure} $W$ associates
to each set $A\subset \Omega$ of finite measure, a Gaussian random variable $W(A)$ so that
\begin{itemize}
  \item $\E W(A)^2 = \nu (A)$,
  \item if $A$ and $A'$ are disjoint sets, then $W(A)$ and $W(A')$ are independent and
  $W(A\cup A') = W(A) + W(A')$.
\end{itemize}
The \emph{stochastic integral} with respect to $W$ is defined by linearly extending
$\int_\Omega 1_A \D W = W(A)$ to simple functions and then by density to whole of $L^2(\Omega)$. Observe, that the
`Itô isometry'
\begin{equation*}
  \E \Big| \int_\Omega u \D W \Big|^2 = \int_\Omega |u|^2 \D\nu
\end{equation*}
holds for $u\in L^2(\Omega)$.
Moreover, if $X$ is
a Banach space, we can take the tensor extension to $L^2(\Omega)\otimes X$ by defining
\begin{equation*}
  \int_\Omega u\otimes \xi \D W = \int_\Omega u \D W \otimes \xi ,
\end{equation*}
for $u\in L^2(\Omega)$ and $\xi\in X$. Two crucial properties of the vector-valued stochastic integral are
\begin{itemize}
\item
\emph{Covariance domination:} If two functions $u,v\in L^2(\Omega) \otimes X$ satisfy
\begin{equation*}
  \int_\Omega |\la v(\cdot ),\xi^* \ra |^2 \D\nu \lesssim \int_\Omega |\la u(\cdot ),\xi^* \ra |^2 \D\nu
\end{equation*}
for all $\xi^*\in X^*$, then
\begin{equation*}
  \E \Big\| \int_\Omega v \D W \Big\|^2 \lesssim \E \Big\| \int_\Omega u \D W \Big\|^2 .
\end{equation*}
\item \emph{Khintchine--Kahane inequality:} For all $1\leq p,q < \infty$ and every $u\in L^2(\Omega) \otimes X$ we have
\begin{equation*}
  \Big( \E \Big\| \int_\Omega u \D W \Big\|^p \Big)^{1/p} 
  \eqsim \Big( \E \Big\| \int_\Omega u \D W \Big\|^q \Big)^{1/q} .
\end{equation*}
\end{itemize}

Recall that a Banach space $X$ is said to have \emph{type} $r\in [1,2]$ if
for any (finite) collection $\{ \xi_k \}$ of vectors in $X$ we have
\begin{equation*}
  \Big( \E \Big\| \sum_k \varepsilon_k \xi_k \Big\|^2 \Big)^{1/2} \lesssim
  \Big( \sum_k \| \xi_k \|^r \Big)^{1/r} ,
\end{equation*}
where the \emph{Rademacher variables} $\varepsilon_k$ are independent and attain values $\pm 1$ with equal probability $1/2$.
In terms of stochastic integrals, if $X$ has type $r$, then
\begin{equation*}
  \Big( \E \Big\| \sum_k \int_\Omega u_k \D W \Big\|^2 \Big)^{1/2} \lesssim
  \Big( \sum_k \E \Big\| \int_\Omega u_k \D W \Big\|^r \Big)^{1/r} ,
\end{equation*}
whenever $u_k$ are disjointly supported functions in $L^2(\Omega) \otimes X$. Indeed, the random variables
$\int_\Omega u_k \D W$ are independent and symmetric, and therefore identically distributed with
$\varepsilon'_k \int_\Omega u_k \D W$ when $(\varepsilon'_k)$ is an independent sequence of Rademacher variables.
Using Khintchine--Kahane inequality and type $r$ of $X$ we may then infer that
\begin{equation*}
  \Big( \E \Big\| \sum_k \int_\Omega u_k \D W \Big\|^2 \Big)^{1/2}
  \eqsim \Big( \E \E' \Big\| \sum_k \varepsilon'_k \int_\Omega u_k \D W \Big\|^r \Big)^{1/r}
  \lesssim \Big( \sum_k \E \Big\| \int_\Omega u_k \D W \Big\|^r \Big)^{1/r} . 
\end{equation*}

The space of `stochastically integrable' functions is not, in general, complete, but can be described in terms of
$\gamma$-radonifying operators (see \cite{GAMMARAD} for a survey):

\begin{def*}
A densely defined linear operator $u$ from $L^2(\Omega)$ to $X$ is said to be \emph{$\gamma$-radonifying} if it can 
be approximated by finite rank operators in the norm
\begin{equation*}
  \| u \|_{\gamma (L^2(\Omega),X)} = \sup \Big( \E \Big\| \sum_k \gamma_k uh_k \Big\|^2 \Big)^{1/2} ,
\end{equation*}
where the supremum is taken over finite orthonormal systems $\{ h_k \}$ in the domain of $u$.
Here the $\gamma_k$ are independent standard Gaussian random variables.
\end{def*}

\begin{remarks}\mbox{}
  \begin{itemize}
  \item Observe that if $\| u \|_{\gamma (L^2(\Omega),X)} < \infty$, then $u$ extends to a bounded operator.
  
  \item 
  If $X$ does not contain an isomorphic 
  copy of $c_0$, then every operator
  $u$ with $\| u \|_{\gamma (L^2(\Omega ),X)} < \infty$ can be approximated by finite rank operators and is thus
  $\gamma$-radonifying \cite[Theorem 4.2]{GAMMARAD}. 
  
  \item The space $\gamma (L^2(\Omega),X)$ of $\gamma$-radonifying operators is complete.
  \end{itemize}
\end{remarks}

Now, $\gamma$-norms of finite rank operators correspond to stochastic integrals of functions in the sense that
every $u = \sum_k u_k \otimes \xi_k \in L^2(\Omega) \otimes X$ defines an operator
\begin{equation*}
  L^2(\Omega) \to X : \quad h\mapsto \sum_k \Big( \int_{\Omega} u_k h \D\nu \Big) \xi_k
\end{equation*}
(also denoted by $u$) for which
\begin{equation*}
  \| u \|_{\gamma (L^2(\Omega),X)} = \Big( \E \Big\| \int_{\Omega} u \D W \Big\|^2 \Big)^{1/2} .
\end{equation*}

\subsection*{The UMD-property}

Most of our results rely on the assumption that $X$ has \emph{UMD}, which by 
definition is a requirement for unconditionality of
martingale differences (see \cite{BURKHOLDERSINGULAR}). It can also be described in terms of various square functions, such as
the Littlewood--Paley square function: $X$ has UMD if and only if for any $1 < p < \infty$ we have
\begin{equation*}
  \E \Big\| \sum_{k\in\Z} \varepsilon_k P_k f \Big\|_{L^p(\R^n ; X)} \eqsim \| f \|_{L^p(\R^n ; X)} ,
\end{equation*}
where $\widehat{P_kf}(\xi) = 1_{A_k}(\xi)\widehat{f}(\xi)$ defines a frequency cut-off to the cubical annulus 
$A_k = \{ \xi\in\R^n : 2^k \leq |\xi_j| < 2^{k+1} \}$. A one-dimensional version of this result first appeared in
\cite{BOURGAINDUALITY} and an extension to higher dimensions can be found in \cite{ZIMMER} 
(see also \cite[Section 4]{WEISBOOK}).
As a consequence
one has the Mihlin multiplier theorem (see \cite[Proposition 3]{ZIMMER} or \cite[4.6 Theorem]{WEISBOOK}) which can be 
applied in showing that
the (non-negative) Laplacian $\Delta$ has
a bounded $H^\infty$-functional calculus on $L^p(\R^n ; X)$, that is, for every
bounded holomorphic function $\phi$ in a sector 
$\{ \zeta\in\C\setminus \{ 0 \} : | \arg \zeta | < \sigma \}$ with $\sigma > 0$,
the Fourier multiplier
\begin{equation*}
  \widehat{\phi (\Delta) f}(\xi) = \phi (|\xi |^2)\widehat{f}(\xi) ,
\end{equation*}
defines a bounded operator $\phi (\Delta)$ on $L^p(\R^n ; X)$. 
On the other hand, boundedness of such functional calculus for the Laplacian on $L^p(\R^n ;X)$ is sufficient for
$X$ to have UMD, as was proven in \cite{GUERRE} by considering the imaginary powers 
arising from $\phi (\zeta) = \zeta^{is}$, with $s\in\R$.
The Mihlin multiplier theorem was extended to the atomic Hardy space
$H^1_{at}(\R^n ; X)$ in \cite{HYTONENEMB} (see page \pageref{hardydef} for the definition).
It should also be mentioned that, 
more generally, any generator of a positive
contraction semigroup on an $L^p$-space has a bounded $H^\infty$-functional calculus on $L^p(X)$ when $X$ has UMD (see
\cite{HIEBER}). The general theory of $H^\infty$-functional calculus for sectorial operators was developed
by McIntosh and collaborators in \cite{MCINTOSH} and \cite{YAGI}.

Our need for UMD is two-fold. In the main example (on page \pageref{example}) we follow 
\cite[Theorem 8.2]{HVNPCONICAL} and make use of vector-valued Calderón--Zygmund theory in studying
$L^p$-boundedness of the conical square function
\begin{equation*}
  Sf(x) = \Big( \E \Big\| \iint_{|x-y|<t} (t^2\Delta)^N e^{-t^2\Delta}f(y) \D W(y,t) \Big\|^2 \Big)^{1/2} ,
\end{equation*}
where $W$ is a random measure arising from $\frac{\textup{d}y \D t}{t^{n+1}}$. In accordance with the discussion above,
this contains the essence of UMD. In addition, we rely on UMD in the form of a vector-valued Stein's inequality,
which is central to our proof of the basic tent space properties (see Proposition \ref{refl} and the references therein).

\section{Tent spaces}
\label{tentspaces}

Let $(M,d,\mu)$ be a complete doubling metric measure space.
This means that there exist a 
number $n>0$ such that for every ball $B\subset M$,
\begin{equation*}
  \mu (\alpha B) \lesssim \alpha^n \mu(B) ,
\end{equation*}
whenever $\alpha \geq 1$. Furthermore, for all $x,y\in M$ and all $r>0$ we have
\begin{equation*}
  \mu (B(x,r)) \lesssim \Big( 1 + \frac{d(x,y)}{r}\Big)^{n_0} \mu (B(y,r)) ,
\end{equation*}
where $0\leq n_0 \leq n$. We fix $n$ and $n_0$ to be smallest such numbers. In what follows, we write
$V(y,t) = \mu (B(y,t))$. By $r_B$ we refer to the radius of a ball $B$.

\subsection*{Definition of and basic properties tent spaces}
We equip the upper half-space $M^+ = M \times (0,\infty )$ with a random measure $W$
arising from $\frac{\textup{d}\mu (y)\,\textup{d}t}{tV(y,t)}$ and write
$\Gamma_\alpha (x) = \{ (y,t)\in M^+ : d(x,y) < \alpha t \}$ for the cone of aperture $\alpha \geq 1$ at $x\in M$.
Note that functions in scalar-valued tent spaces,\footnote{Familiarity with the basics of scalar-valued tent spaces is assumed; see \cite{AMENTA,CMSTENTSPACES}.} being locally square-integrable, can be seen to act as linear functionals on
the space $L_c^2(M^+)$ of compactly supported square-integrable functions on $M^+$.
It is therefore natural to define vector-valued tent spaces to consist of linear operators
from $L^2_c(M^+)$ to $X$.
We use $1_K$ synonymously for the indicator function and the corresponding projection operator.
Integration on $M^+$ is denoted by the double integral $\iint$ and
integral averages on $M$ are abbreviated by 
$\fint_B \textup{d}\mu := \mu (B)^{-1} \int_B \textup{d}\mu$.
Let $X$ be a (complex) Banach space.

\begin{def*}
  Let $1\leq p < \infty$ and $\alpha \geq 1$. The tent space $T_\alpha^p(X)$ consists of linear operators
  $u: L^2_c(M^+) \to X$ for which
  \begin{itemize}
    \item the map $x\mapsto u 1_{\Gamma_\alpha (x)}$ is strongly measurable from $M$ to $\gamma (L^2(M^+),X)$,
    \item $\| u \|_{T_\alpha^p(X)} = \| \Aa_\alpha u \|_{L^p} < \infty$, where
    $\Aa_\alpha u(x) = \| u 1_{\Gamma_\alpha (x)} \|_{\gamma (L^2(M^+),X)}$. 
  \end{itemize}
\end{def*}

\begin{remarks}\mbox{}
\begin{itemize}
\item For every $1\leq p < \infty$ and $\alpha \geq 1$, the tent space $T_\alpha^p(X)$ is complete and 
      contains $L^2_c(M^+) \otimes X$ as a dense subspace 
      (see Appendix A). From Propositions \ref{refl} and 
      \ref{changeofaperture} it follows that, under our typical assumptions on $X$ and $M$,
      the tent spaces with different apertures $\alpha$ coincide for any fixed $1\leq p < \infty$.

\item Let $1\leq p < \infty$. Note that if $u \in T^p$ and $\xi\in X$, then
\begin{equation*}
  \Aa (u\otimes \xi )(x) = \Big( \E \Big\| \iint_{\Gamma (x)} u \D W \otimes \xi \Big\|^2 \Big)^{1/2}
  = \Big( \iint_{\Gamma (x)} |u(y,t)|^2 \, \frac{\textup{d}\mu (y) \D t}{t V(y,t)} \Big)^{1/2} \| \xi \| ,
\end{equation*}
and so $T^p \otimes X$ is a dense subspace of $T^p(X)$.
Here and in what follows, by omitting the parameter 
$\alpha$ we refer to $\alpha = 1$.

\item The most fundamental difference to the scalar-valued tent spaces is that, unless $X$ is a Hilbert space,
we no longer have $T^2(X) = L^2(M^+, \frac{\textup{d}\mu \D t}{t} ; X)$.
\end{itemize}
\end{remarks}

For $x\in M$ and $r > 0$, let $\Gamma^r (x) = \{ (y,t)\in \Gamma (x) : t < r \}$ denote a truncated cone.

\begin{def*}
  The tent space $T^\infty (X)$ consists of linear operators $v: L^2_c(M^+) \to X$ for which
  \begin{itemize}
    \item the map $x\mapsto v 1_{\Gamma^r (x)}$ is strongly measurable from $M$ to $\gamma (L^2(M^+),X)$ for every $r > 0$,
    \item the norm
    \begin{equation*}
      \| v \|_{T^\infty (X)} = \sup_B \Big( \fint_B \Aa^{r_B} v(x)^2 \D\mu (x) \Big)^{1/2} < \infty ,
    \end{equation*}
    where $\Aa^r v(x) = \| v 1_{\Gamma^r (x)} \|_{\gamma (L^2(M^+),X)}$ and the supremum is taken over all balls $B\subset M$. 
  \end{itemize}
\end{def*}

\begin{remark}
  For scalar-valued functions the $T^\infty$-norm is comparable with a more familiar expression. Indeed, if
  $v\in T^\infty$ and $\xi \in X$, then
  \begin{align*}
    \| v \otimes \xi \|_{T^\infty (X)}
    &= \sup_B \Big( \fint_B \iint_{\Gamma^{r_B}(x)} |v(y,t)|^2 \, \frac{\textup{d}\mu (y) \D t}{tV(y,t)} \D\mu (x) \Big)^{1/2} 
    \| \xi \| \\
    &\eqsim \sup_B \Big( \frac{1}{\mu (B)} \iint_{T(B)} |v(y,t)|^2 \, \frac{\textup{d}\mu (y) \D t}{t}\Big)^{1/2}
    \| \xi \| ,
  \end{align*}
  where we made use of the observation that for each ball $B\subset M$ and every $x\in B$ 
  we have $\Gamma^{r_B}(x) \subset T(3B) := M^+ \setminus \bigcup_{x\not\in 3B} \Gamma (x)$. 
  Consequently, $T^\infty \otimes X$ is a subspace of $T^\infty (X)$ (but not dense).
\end{remark}

The following proposition presents three basic properties of tent spaces in the case $1 < p < \infty$. An efficient way to
handle this range by embedding into vector-valued $L^p$-spaces was discovered in \cite{HARBOURE}.

\begin{prop}
\label{refl}
  Let $1 < p < \infty$ and suppose that $X$ has UMD.
  \begin{itemize}
    \item \emph{Change of aperture:} for every
    $u\in L^2_c(M^+)\otimes X$ we have 
    $\| \Aa_\alpha u \|_{L^p} \lesssim_p \alpha^n \| \Aa u \|_{L^p}$ whenever $\alpha \geq 1$.
    \item \emph{Duality:} the isomorphism $T^p(X)^* \simeq T^{p'}(X^*)$ is realized by the pairing
    \begin{equation*}
      \la u , v \ra = \iint_{M^+} \la u(y,t) , v(y,t) \ra \, \frac{\textup{d}\mu (y) \D t}{t} , 
      \quad u\in T^p \otimes X , \quad v\in T^{p'} \otimes X^*,
    \end{equation*}
    for which $| \la u , v \ra | \lesssim \| u \|_{T^p(X)} \| v \|_{T^{p'}(X^*)}$.
    \item \emph{Complex interpolation:} we have $[T^{p_0}(X) , T^{p_1}(X)]_\theta = T^p(X)$, where $1 < p_0 \leq p_1 < \infty$ and $1/p = (1-\theta)/p_0 + \theta /p_1$.
  \end{itemize}
  \begin{proof}
  We content ourselves with a sketch of the proof. For more details, see \cite{HVNPCONICAL,TENTSPACES} and the references 
  therein.
  The isometry
  \begin{equation*}
    J_\alpha : T^p_\alpha (X) \hookrightarrow L^p(M;\gamma (L^2(M^+),X)) , \quad J_\alpha u(x) = u1_{\Gamma_\alpha (x)}
  \end{equation*}
  embeds $T^p_\alpha (X)$ as a complemented subspace of $L^p(M;\gamma (L^2(M^+),X))$. The associated projection is
  given by
  \begin{equation*}
    N_\alpha F(x;y,t) = 1_{B(x,\alpha t)} (x) \fint_{B(y,t)} F(z;y,t) \D\mu (z) , \quad F\in L^p(M) \otimes L^2(M^+) \otimes X .
  \end{equation*}
  Note that $N_\alpha F(x;y,t) = A_{B(y,t)}^\alpha F_{y,t}(x)$, where $F_{y,t}$ stands for the function
  $M\to X : x\to F(x;y,t)$ and
  \begin{equation*}
    A_B^\alpha f = 1_{\alpha B} \fint_B f\D\mu
  \end{equation*}
  is a localized averaging operator associated with a ball $B\subset M$. Consequently,
  \begin{equation*}
    \| N_\alpha F \|_{L^p(M;\gamma (L^2(M^+),X))} 
    \lesssim \gamma (A_B^\alpha : B\subset M) \| F \|_{L^p(M;\gamma (L^2(M^+),X))} ,
  \end{equation*}
  where $\gamma (\cdot)$ is the \emph{$\gamma$-bound} of the family $\{ A_B^\alpha \}_{B\subset M}$ on $L^p(M;X)$, i.e.
  the smallest constant $C_p$ so that
  \begin{equation*}
    \E \Big\| \sum_k \gamma_k A_{B_k}^\alpha f_k \Big\|_{L^p(M;X)}^2 \leq C_p^2
    \E \Big\| \sum_k \gamma_k f_k \Big\|_{L^p(M;X)}^2
  \end{equation*}
  for any (finite) collections of balls $B_k\subset M$ and functions $\{ f_k \} \subset L^p(M;X)$.

  In order to calculate the $\gamma$-bound, we approximate $A_B^\alpha$ by dyadic averaging operators.
  Recall that a \emph{dyadic system} on a $M$ is a collection $\Dd = \{\Dd_k\}_{k \in \Z}$, where each $\Dd_k$ is a 
  partition of $M$ into sets of finite positive measure, such that the containment relations
  \begin{equation*}
    Q\in\Dd_k , \quad Q'\in\Dd_{k'} , \quad k'\geq k \quad \Longrightarrow
    \quad Q'\subset Q \quad \text{or} \quad Q\cap Q' = \emptyset
  \end{equation*}
  hold. By Stein's inequality (see \cite[Lemma 3.1]{HVNPCONICAL} and
  the references therein), the families $\{ A_Q \}_{Q\in\Dd}$ of localized dyadic averaging operators
  \begin{equation*}
    A_Q f = 1_Q \fint_Q f\D\mu
  \end{equation*}
  are $\gamma$-bounded on $L^p(M;X)$ when $1 < p < \infty$.

  In \cite{HYTONENKAIREMA} it is shown 
  that one can choose a finite number of dyadic systems
  on $M$ so that every ball $B\subset M$ is contained in a dyadic cube $Q_B$ from one of the dyadic systems, with 
  $\textup{diam}\,(Q_B) \lesssim \textup{diam}\,(B)$. Therefore we may write
  \begin{equation*}
    A_B^\alpha f = 1_{\alpha B} \frac{\mu (Q_{\alpha B})}{\mu (B)} A_{Q_{\alpha B}} (1_Bf) ,
  \end{equation*}
  and hence
  \begin{equation*}
    \gamma (A_B^\alpha : B\subset M) \lesssim \frac{\mu (Q_{\alpha B})}{\mu (B)} \lesssim \alpha^n ,
  \end{equation*}
  with a constant depending on $p$.
  
  The claim of change of aperture now follows from the identity $J_\alpha u = N_\alpha J u$. Duality and complex interpolation
  follow from the corresponding results for complemented subspaces of vector-valued $L^p$-spaces.
  \end{proof}
\end{prop}

\begin{remark}
  It should be pointed out that in the proof above the 
  $\gamma$-bounds of the families 
  $\{ A_Q \}_{Q\in\Dd}$ and $\{ A_B^\alpha \}_{B\subset M}$ 
  on $L^p(M;X)$ tend to infinity as $p\to 1$, and, therefore, so does
  the $p$-dependent constant obtained by this method 
  for the change of aperture.
\end{remark}

\subsection*{Cone covering property}
We now elaborate the additional geometric assumption on $M$ (originating from \cite{TENTSPACES}), which we use to extend
Proposition \ref{refl} to the endpoint $p=1$.
Given a $\sigma\in (0,1)$ we define the extension of an open set $E\subset M$ by
\begin{equation*}
  E^\sigma = \{ x\in M : \sup_{B\ni x} \frac{\mu (B\cap E)}{\mu (B)} > \sigma \} .
\end{equation*}
Note that $E^\sigma$ is open and satisfies $\mu (E^\sigma) \lesssim \sigma^{-1}\mu (E)$ by the weak type $(1,1)$ inequality
for the Hardy--Littlewood maximal function.
Recall that the tent $T(E)$ over an open set $E\subset M$ is given by
\begin{equation*}
  T(E) = \{ (y,t)\in M^+ : B(y,t)\subset E \} 
  = M^+ \setminus \bigcup_{x\not\in E} \Gamma (x) .
\end{equation*}

\begin{cone*}
  There exists a $\sigma\in (0,1)$ such that every bounded open set $E\subset M$ satisfies the following: 
  For every $x\in E$ there exist $x_1,\ldots , x_N\in M\setminus E$, with $N$ depending only on $M$, such that
  \begin{equation*}
    \Gamma (x) \setminus T(E^\sigma) \subset \bigcup_{m=1}^N \Gamma (x_m) .
  \end{equation*}
\end{cone*}

When $M$ has the cone covering property, $\sigma$ will be fixed and we write $E^\sigma = E^*$.

\begin{lemma}
  \label{pointwise}
  Suppose that $M$ has the cone covering property. Let $u\in L^2_c(M^+) \otimes X$ and write
  $E = \{ x\in M : \Aa u(x) > \lambda \}$ for a $\lambda > 0$. Then
  \begin{equation*}
    \Aa (u1_{M^+\setminus T(E^*)})(x) \lesssim \lambda \quad \text{for all $x\in M$} .
  \end{equation*}
  \begin{proof}
    If $x\in M\setminus E$, then
    \begin{equation*}
      \Aa (u1_{M^+\setminus T(E^*)})(x) \leq \Aa u(x) \leq \lambda
    \end{equation*}
    by the definition of $E$.
    Let then $x\in E$.
    Since $E$ is a bounded open set, we may use the cone covering property to pick
    $x_1,\ldots , x_N\in X\setminus E$ (with $N$ depending only on the dimension of $M$) such that
    \begin{equation*}
      \Gamma (x) \setminus T(E^*) \subset \bigcup_{m=1}^N \Gamma (x_m) .
    \end{equation*}
    We can then estimate
    \begin{equation*}
      \Aa (u1_{M^+\setminus T(E^*)})(x)
      = \Big( \E \Big\| \iint_{\Gamma (x) \setminus T(E^*)} u \D W \Big\|^2 \Big)^{1/2}
      \leq \sum_{m=1}^N \Big( \E \Big\| \iint_{\Gamma (x_m)} u \D W \Big\|^2 \Big)^{1/2} 
      \leq N\lambda,
    \end{equation*}
    as required.
  \end{proof}
\end{lemma}

\begin{remark}
  In \cite[Appendix B]{NONNEGATIVE} we have shown that every complete (connected) Riemannian manifold with non-negative sectional curvature 
  has the cone covering property.
  The lemma above should be compared with \cite[Lemma 4.4]{NONNEGATIVE}. Notice, that
  in the vector-valued setting, Bernal's convex reduction argument \cite{BERNAL} 
  is not available, which means that interpolation and
  change of aperture for $T^1(X)$ cannot be deduced from the reflexive range as in the scalar-valued case, 
  and this forces us to use the cone covering property.
\end{remark}

\subsection*{Atomic decomposition}

The main result of \cite{TENTSPACES} was the atomic decomposition for $T^1(X)$ on $\R^n$, which also relies
on the cone covering property. The proof generalizes directly to our setting.

\begin{def*}
  An $a\in T^1(X)$ is called an \emph{atom} associated with ball $B\subset M$ if $a1_{T(B)} = a$
  (i.e. $a$ is `supported' in $T(B)$) and $\| a \|_{T^2(X)} \leq \mu (B)^{-1/2}$.
\end{def*}

\begin{theorem}[Atomic decomposition]
\label{atomicdec}
  Suppose that $M$ has the cone covering property. Then
  every $u\in T^1(X)$ can be decomposed into atoms $a_k$ so that
  \begin{equation*}
    u = \sum_k \lambda_k a_k ,
  \end{equation*}
  where the sum converges in $T^1(X)$ and the scalars $\lambda_k$
  satisfy
  \begin{equation*}
    \sum_k |\lambda_k| \eqsim \| u \|_{T^1(X)} .
  \end{equation*}
  Moreover, if $u\in (T^1 \cap T^2) \otimes X$, then the sum converges also in
  $T^2(X)$.
\end{theorem}

This allows us to extend the change of aperture estimate from Proposition \ref{refl} to $T^1(X)$.

\begin{prop}
\label{changeofaperture}
  Suppose that $X$ has UMD and that $M$ has the cone covering property. Let $\alpha \geq 1$. Then, given any $\varepsilon > 0$, we have
  \begin{equation*}
    \| \Aa_\alpha u \|_{L^1} \lesssim_\varepsilon \alpha^{n+\varepsilon} \| \Aa u \|_{L^1}
  \end{equation*}
  for every $u\in L^2_c(M^+) \otimes X$.
  \begin{proof}
    Note first that if $a$ is an atom associated with a ball $B\subset M$,
    then $\| a \|_{T^p(X)} \leq \mu (B)^{-(1-1/p)}$ for $1 \leq p \leq 2$
    as an immediate consequence of $\| a \|_{T^1(X)} \leq 1$.
    Secondly, for any ball $B$, $\Gamma_\alpha (x)$ intersects $T(B)$ exactly
    when $x\in \alpha B$. Thus, given an $\varepsilon > 0$ we may write
    $1 - 1/p = \varepsilon$ with a $p > 1$ and argue as follows:
    \begin{align*}
      \| \Aa_\alpha a \|_{L^1} &= \int_{\alpha B} \Aa_\alpha a(x) \D\mu (x)
      \leq \mu (\alpha B)^{1 - 1/p} \Big( \int_{\alpha B} \Aa_\alpha a(x)^p \D\mu (x) \Big)^{1/p} \\
      &\lesssim_p \mu (\alpha B)^{1 - 1/p} \alpha^n \| a \|_{T^p(X)}
      \leq \Big( \frac{\mu (\alpha B)}{\mu (B)} \Big)^{1 - 1/p} \alpha^n 
      = \alpha^{n+\varepsilon} ,
    \end{align*}
    where in the third step we used Proposition \ref{refl}. 
    The claim follows by the Atomic decomposition. 
  \end{proof}
\end{prop}

\begin{theorem}
\label{endpointduality}
  Suppose that $X$ has UMD and that $M$ has the cone covering property. Then $T^1(X)^* = T^\infty (X^*)$.
  \begin{proof}
    To see that every $v\in T^\infty (X^*)$ induces a bounded linear functional $\Lambda$ on $T^1(X)$, note first that
    for any ball $B\subset M$,
    \begin{align*}
      \| v1_{T(B)} \|_{T^2(X^*)} &= \Big( \int_B \Aa (v1_{T(B)})(x)^2 \D\mu (x) \Big)^{1/2} \\
      &\leq \Big( \int_B \Aa^{r_B} v(x)^2 \D\mu (x) \Big)^{1/2} \leq \mu (B)^{1/2} \| v \|_{T^\infty (X^*)} .
    \end{align*}
    By the Atomic decomposition, it suffices to define the action of $\Lambda$ on atoms: if $a$ is an atom in $T(B)$ we set
    $\Lambda a = \la a , v1_{T(B)} \ra$ so that
    \begin{equation*}
      |\Lambda a | \leq |\la a , v1_{T(B)} \ra | \leq \| a \|_{T^2(X)} \| v1_{T(B)} \|_{T^2(X^*)}
      \leq \| v \|_{T^\infty (X^*)} .
    \end{equation*}
    This does not depend on $B$ in the sense that if $a$ is an atom in both $T(B)$ and $T(B')$, then
    $\la a , v1_{T(B)} \ra = \la a , v1_{T(B')} \ra$.
  
    Let $\Lambda\in T^1(X)^*$. For every open $E\subset M$ we have $\Gamma (x) \cap T(E) \neq \emptyset$ exactly when
    $x\in E$ so that $\Aa (u1_{T(E)})$ is supported in $E$ and 
    $\| u1_{T(E)} \|_{T^1(X)} \leq \mu (E)^{1/2} \| u1_{T(E)} \|_{T^2(X)}$ whenever $u\in T^2(X)$.
    Hence $\Lambda$ restricts to a bounded linear functional $\Lambda_E$ on the closed (complemented) subspace 
    $T^2_E(X) = \{ u1_{T(E)} : u\in T^2(X) \}$ of $T^2(X)$. Since $X$ has UMD, $T^2_E(X)^* = T^2_E(X^*)$ (by Proposition
    \ref{refl}) and there exists a $v_E\in T^2_E(X^*)$ so that $\Lambda_E u = \la u , v_E \ra$ for all $u\in T^2_E(X)$ and
    \begin{equation*}
      \| v_E \|_{T^2(X^*)} \eqsim \| \Lambda_E \|_{T^2_E(X)^*} \leq \mu (E)^{1/2} \| \Lambda \|_{T^1(X)^*} .
    \end{equation*}
    
    Moreover, $v_E 1_{T(E\cap E')} = v_{E'} 1_{T(E\cap E')}$ because for
    every $u\in T^2(X)$ we have
    $\la u, v_E 1_{T(E\cap E')} \ra = \Lambda (u1_{T(E\cap E')}) = \la u, v_{E'} 1_{T(E\cap E')} \ra$. Consequently,
    $v_E h = v_{E'}h$ for all $h\in L^2(K)$ whenever $K\subset T(E\cap E') = T(E) \cap T(E')$ and we may define
    a linear operator $v : L^2_c(M^+) \to X$ by $vh = v_Eh$ when $h\in L^2(K)$ with $K\subset T(E)$.
    
    To see that $\| v \|_{T^\infty (X^*)} \eqsim \| \Lambda \|_{T^1(X)^*}$ note first that    
    for any ball $B\subset M$, we have $\Gamma (x;r_B) \subset T(3B)$ whenever $x\in B$. Therefore
    \begin{equation*}
      \Big( \fint_B \Aa^{r_B}v(x)^2 \D\mu (x) \Big)^{1/2} \leq \frac{1}{\mu (B)^{1/2}} \Big( \int_B \Aa (v_{3B})(x)^2 \D\mu (x)
      \Big)^{1/2} \leq \frac{\| v_{3B} \|_{T^2(X^*)}}{\mu (B)^{1/2}}  \lesssim \| \Lambda \|_{T^1(X)^*} ,
    \end{equation*}
    and so $\| v \|_{T^\infty (X^*)} \lesssim \| \Lambda \|_{T^1(X)^*}$.
    On the other hand, by the Atomic decomposition, $\| \Lambda \|_{T^1(X)^*}$ is obtained by testing against atoms.
    Now, if $a$ is an atom in $T(B)$, then
    \begin{align*}
      |\Lambda a | = |\la a , v_B \ra | \leq \| a \|_{T^2(X)} \| v_B \|_{T^2(X^*)}
      &\leq \frac{1}{\mu (B)^{1/2}} \Big( \int_B \Aa v_B(x)^2 \D\mu (x) \Big)^{1/2} \\
      &\leq \Big( \fint_B \Aa^{r_B}v(x)^2 \D\mu (x) \Big)^{1/2} \leq \| v \|_{T^\infty (X^*)} .
    \end{align*}
  \end{proof}
\end{theorem}

\begin{remark}
  That every $v\in T^\infty (X^*)$ induces a bounded linear functional on $T^1(X)$ follows also from the inequality
  \begin{equation*}
    \iint_{M^+} |\la u(y,t),v(y,t) \ra | \, \frac{\textup{d}\mu(y)\D t}{tV(y,t)} \lesssim \| u \|_{T^1(X)}
    \| v \|_{T^\infty (X^*)} , \quad u\in T^1 \otimes X ,
  \end{equation*}
  where $v$ is assumed to be a function. This can be proved as in \cite{HYTONENWEISPARAPRODUCTS} and \cite{CMSTENTSPACES}.
\end{remark}

\subsection*{Interpolation}
Our first main result extends the complex interpolation scale of vector-valued tent spaces 
\cite[Theorem 4.7]{HVNPCONICAL} to the endpoint $p=1$. The argument
presented here fills the gap in the proof of \cite[Lemma 5]{CMSTENTSPACES}
(see also \cite[Remark 3.20]{AMENTA}) by using the cone covering property.

\begin{theorem}
\label{interpolationendpoint}
  Suppose that $X$ has type $r \in (1,2]$ and that $M$ has the cone covering property. Then
  \begin{equation*}
    [T^1(X),T^r(X)]_\theta = T^p(X) , \quad \textup{where}\quad \frac{1}{p} = 1- \theta (1 - \frac{1}{r}) .
  \end{equation*}
  \begin{proof}
    We first check that $[T^1(X),T^r(X)]_\theta \subset T^p(X)$.
    Let $\Upsilon : \overline{S} \to T^1(X) + T^r(X)$ be a function that\footnote{The reader is referred to \cite[Chapter 4]{BERGH} for details on complex interpolation.}
    \begin{itemize}
    \item is analytic in the strip $S = \{ \zeta\in\C : 0 < \textup{Re}\, \zeta < 1 \}$, 
    \item is continuous and bounded on $\overline{S}$,
    \item has $\| \Upsilon (is) \|_{T^1(X)} \lesssim 1$ and $\| \Upsilon (1+is) \|_{T^r(X)} \lesssim 1$ for all $s\in\R$.
    \end{itemize}
    Denote $Y=\gamma (L^2(M^+),X)$ and
    recall the embedding
    $T^p(X) \hookrightarrow L^p(M;Y)$ given by $Ju(x) = u1_{\Gamma (x)}$.
    Then 
    $J\circ \Upsilon : \overline{S} \to L^1(M ; Y) + L^r(M ; Y)$    
    and we may rely on complex interpolation for vector-valued $L^q$-spaces to see that
    \begin{align*}
      \| \Upsilon (\theta) \|_{T^p(X)} &= \| J\circ \Upsilon (\theta) \|_{L^p(M ; Y)} \\
      &\leq \max \Big\{ \sup_{s\in\R} \| J\circ \Upsilon (is) \|_{L^1(M ; Y)} , \;
      \sup_{s\in\R} \| J\circ \Upsilon (1+is) \|_{L^r(M ; Y)} \Big\} \\
      &= \max \Big\{ \sup_{s\in\R} \| \Upsilon (is) \|_{T^1(X)} , \; \sup_{s\in\R} \| \Upsilon (1+is) \|_{T^r(X)} \Big\} ,
    \end{align*}
    which shows that $[T^1(X),T^r(X)]_\theta$ is boundedly contained in $T^p(X)$.
    
    We now show that $[T^1(X),T^r(X)]_\theta \supset T^p(X)$:
    Let $u\in L^2_c(M^+) \otimes X$ with $\| u \|_{T^p(X)} = 1$ and consider the open sets
    \begin{equation*}
      E_k = \{ x\in M : \Aa u(x) > 2^k \} , \quad k\in\Z .
    \end{equation*}
    Write $A_k = T(E_k^*) \setminus T(E_{k+1}^*)$ and define the interpolating function as in \cite[Lemma 5]{CMSTENTSPACES} by
    \begin{equation*}
      \Upsilon (\zeta) = \sum_{k\in\Z} 2^{k(\upsilon (\zeta)p - 1)} u1_{A_k} , \quad \text{where} \quad
      \upsilon (\zeta) = 1 - \zeta (1 - \frac{1}{r}),
    \end{equation*}
    so that $\Upsilon (\theta ) = u$. What remains is to check that 
    $\| \Upsilon (is) \|_{T^1(X)} \lesssim 1$ and
    $\| \Upsilon (1 + is) \|_{T^r(X)} \lesssim 1$ for all $s\in\R$.
    
    Let $s\in\R$ and note first that $|2^{k(\upsilon (is)p-1)}| \leq 2^{k(p-1)}$.
    Hence by triangle inequality
    \begin{equation*}
      \| \Upsilon (is) \|_{T^1(X)} \leq \sum_{k\in\Z} 2^{k(p-1)} \| u1_{A_k} \|_{T^1(X)} ,
    \end{equation*}
    where
    \begin{equation*}
      \| u1_{A_k} \|_{T^1(X)} = \int_{E_k^*} \Aa (u1_{A_k})(x) \D\mu (x) \lesssim 2^k \mu (E_k^*) ,
    \end{equation*}
    according to Lemma \ref{pointwise}.
    Consequently,
    \begin{equation*}
      \| \Upsilon (is) \|_{T^1(X)} \lesssim \sum_{k\in\Z} 2^{kp} \mu (E_k^*) \lesssim \| u \|_{T^p(X)}^p .
    \end{equation*}
    
    For a given $s\in\R$ we now estimate the second quantity 
    \begin{equation*}
      \| \Upsilon (1+is) \|_{T^r(X)}^r = \int_M \Big( \E \Big\| \iint_{\Gamma (x)}
      \sum_{k\in\Z} 2^{k(\upsilon (1+is)p-1)} u1_{A_k} \D W \Big\|^2 \Big)^{r/2} \D\mu (x) .
    \end{equation*}
    Noting that $|2^{k(\upsilon (1+is)p-1)}| \leq 2^{k(p/r - 1)}$ we argue using type $r$ of $X$:
    \begin{equation*}
      \Big( \E \Big\| \sum_{k\in\Z} \iint_{\Gamma (x)} 
      2^{k(\upsilon (1+is)p-1)} u1_{A_k} \D W \Big\|^2 \Big)^{1/2}
      \leq \Big( \sum_{k\in\Z} 2^{k(p-r)} \E \Big\| \iint_{\Gamma (x)} u1_{A_k} \D W \Big\|^r \Big)^{1/r} .
    \end{equation*}
    Therefore, by Lemma \ref{pointwise},
    \begin{align*}
      \| \Upsilon (1+is) \|_{T^r(X)}^r &\lesssim \sum_{k\in\Z} 2^{k(p-r)} 
      \int_{E_k^*} \E \Big\| \iint_{\Gamma (x)} u1_{A_k} \D W \Big\|^r \D\mu (x) \\
      &\lesssim \sum_{k\in\Z} 2^{k(p-r)} \int_{E_k^*} \Aa (u1_{A_k})(x)^r \D\mu (x) \\
      &\lesssim \sum_{k\in\Z} 2^{kp} \mu (E_k^*) 
      \lesssim \| u \|_{T^p(X)}^p ,
    \end{align*}
    as required.
\end{proof}
\end{theorem}

\begin{remark}
  It is clear that for $1\leq p < \infty$, the tent spaces $T^p(X)$ embed continuously into
  $L^1_{loc}(M;\gamma (L^2(M^+),X))$. Another possible choice for an ambient space, one that is suitable also for
  $T^\infty (X)$, is the space of linear operators $u : L^2(M^+) \to X$ equipped with the seminorms
  $\| u1_K \|_{\gamma (L^2(M^+),X)}$ with $K\subset M^+$ ranging over compact subsets of $M^+$.
\end{remark}

\begin{cor}[Complex interpolation]
\label{tentmain1}
  Suppose that $X$ has UMD and that $M$ has the cone covering property. Let $1\leq p_0 \leq p_1 \leq \infty$. Then
  \begin{equation*}
    [T^{p_0}(X),T^{p_1}(X)]_\theta = T^p(X) , \quad \text{where} \quad
    \frac{1}{p} = \frac{1-\theta}{p_0} + \frac{\theta}{p_1} .
  \end{equation*}
  \begin{proof}
    By Proposition \ref{refl} the claim is true for $1 < p_0 \leq p_1 < \infty$.
    First, take $r > 1$ so that $X$ has type $r$. 
    The statement then follows for $p_0=1$ and $p_1=r$ from Theorem \ref{interpolationendpoint}.
    For $p_0=2$ and $p_1=\infty$ we argue by duality. Note that $1/p = (1-\theta ) /2$ implies that
    $1/p' = 1 - \theta' + \theta' / 2$ for $\theta' = 1-\theta$. Then
    \begin{equation*}
      [T^2(X) , T^\infty (X)]_\theta = [T^1(X^*) , T^2(X^*)]_{\theta'}^* = T^{p'}(X^*)^* = T^p(X)
    \end{equation*}
    by reflexivity of $X$ and Proposition \ref{refl}.
    The full statement now follows by reiteration (and its converse). 
  \end{proof}
\end{cor}

\subsection*{Integral operators on tent spaces}
We will then consider integral operators on tent spaces. Given an operator-valued kernel
$K : (0,\infty) \times (0,\infty) \to \Ll (L^2(M))$ we define
\begin{equation*}
  Su(\cdot , t) = \int_0^\infty K(t,s)u(\cdot , s) \Ds , \quad t > 0, \quad u\in L^2_c(M^+) \otimes X .
\end{equation*}
The following result extends \cite[Corollary 5.1]{HVNPCONICAL} to $T^1(X)$. In the statement and the proof, 
the only difference to the Euclidean setting is that
we might no longer have $\mu (B(x,t)) \eqsim t^n$, and therefore have to assume more decay from the kernel.

\begin{theorem}
\label{intop}
  Suppose that $X$ has UMD and that $M$ has the cone covering property. Assume that the kernel satisfies
  for all $t,s>0$ the estimate
  \begin{equation}
  \label{OD}
    \| 1_{E'} K(t,s) (1_Ef) \|_{L^2} 
    \lesssim \min \Big( \frac{t^\alpha}{s^\alpha} , \frac{s^\beta}{t^\beta} \Big)
    \Big( 1 + \frac{d(E,E')}{\max (t,s)} \Big)^{-\gamma} \| 1_Ef \|_{L^2}
  \end{equation}
  whenever $E,E' \subset M$ are measurable and $f\in L^2(M)$, and that
  $\gamma > 3n/2$ and $\alpha , \beta > n$. Then $S$ is bounded on $T^p(X)$ for every $1\leq p < \infty$.
  \begin{proof}
    Let $u\in L^2_c(M^+) \otimes X$.
    We closely follow the proofs of \cite[Propositions 5.4 and 5.5]{HVNPCONICAL} and split the operator $S$ into two parts
    \begin{equation*}
      S_\infty u(\cdot , t) = \int_t^\infty K(t,s)u(\cdot , s) \Ds \quad \text{and} \quad
      S_0 u(\cdot , t) = \int_0^t K(t,s)u(\cdot , s) \Ds .
    \end{equation*}

    \emph{The operator $S_\infty$}:
    We estimate $\Aa (S_\infty u)$ pointwise by a sum of $\Aa_{2^{k+1}} u$'s.    
    In order to do this, fix an $x\in M$ and write
    \begin{equation*}
      S_\infty u(\cdot , t) = \sum_{k=0}^\infty \int_t^\infty K(t,s) ( 1_{C_k(x,s)} u(\cdot , s)) \Ds 
      =: \sum_{k=0}^\infty u_k(\cdot , t) ,
    \end{equation*}
    where $C_k(x,s) = B(x,2^{k+1}s) \setminus B(x,2^ks)$ for $k\geq 1$ and $C_0(x,s) = B(x,2s)$.
    The desired estimate
    \begin{equation*}
      \Big( \E \Big\| \iint_{\Gamma (x)} u_k \D W \Big\|^2 \Big)^{1/2} \lesssim 2^{-k\delta}
      \Big( \E \Big\| \iint_{\Gamma_{2^{k+1}} (x)} u \D W \Big\|^2 \Big)^{1/2} ,
    \end{equation*}
    with $\delta > 0$ follows by Covariance domination once we have established that for all $\xi^*\in X^*$,
    \begin{equation*}
      \Big( \iint_{\Gamma (x)} | \la u_k(y,t),\xi^* \ra |^2 \, \frac{\textup{d}\mu (y) \D t}{t V(y,t)} \Big)^{1/2}
      \lesssim 2^{-k\delta} \Big( \iint_{\Gamma_{2^{k+1}}(x)} | \la u(y,t),\xi^* \ra |^2 \, 
      \frac{\textup{d}\mu (y) \D t}{t V(y,t)} \Big)^{1/2} ,
    \end{equation*}
    where
    \begin{equation*}
      \la u_k(\cdot , t) , \xi^* \ra = \Big\la \int_t^\infty K(t,s) ( 1_{C_k(x,s)} u(\cdot , s)) \Ds , \xi^* \Big\ra
      = \int_t^\infty K(t,s) ( 1_{C_k(x,s)} \la u(\cdot , s) , \xi^* \ra ) \Ds .
    \end{equation*}
    
    For a fixed $\xi^*\in X^*$ denote $\hat{u}(\cdot ,s) = \la u( \cdot ,s) , \xi^* \ra$.
    When $(y,t)\in\Gamma (x)$ we have $V(y,t) \eqsim V(x,t)$ and so
    \begin{align*}
      I_k(x) :&= \Big( \iint_{\Gamma (x)} \Big| \int_t^\infty K(t,s)
      ( 1_{C_k(x,s)} \hat{u}(\cdot , s))(y) \Ds \Big|^2 \,\frac{\textup{d}\mu (y) \,\textup{d}t}{tV(y,t)}
      \Big)^{1/2} \\
      &\lesssim \Big( \int_0^\infty \Big( \int_t^\infty \| 1_{B(x,t)} K(t,s) ( 1_{C_k(x,s)} \hat{u}(\cdot , s) )
      \|_{L^2} \Ds \Big)^2 \,\frac{\textup{d}t}{tV(x,t)} \Big)^{1/2} .
    \end{align*}
    For $s>t$ we have $d(B(x,t),C_k(x,s)) \gtrsim 2^ks$ (when $k\geq 1$) and so by \eqref{OD},
    \begin{equation*}
      \| 1_{B(x,t)} K(t,s)(1_{C_k(x,s)}\hat{u}(\cdot ,s )) \|_{L^2}
      \lesssim \Big( \frac{t}{s} \Big)^\alpha 2^{-k\gamma} \| 1_{B(x,2^{k+1}s)} \hat{u}(\cdot , s) \|_{L^2} .
    \end{equation*}
    Therefore
    \begin{align*}
      &\Big( \int_t^\infty \| 1_{B(x,t)} K(t,s) ( 1_{C_k(x,s)} \hat{u}(\cdot , s) )
      \|_{L^2} \Ds \Big)^2 \\
      &\lesssim \int_t^\infty \Big( \frac{t}{s} \Big)^{2\varepsilon} \Ds
      \int_t^\infty \Big( \frac{t}{s} \Big)^{2(\alpha - \varepsilon)} 4^{-k\gamma} 
      \| 1_{B(x,2^{k+1}s)} \hat{u}(\cdot , s) \|_{L^2}^2
      \Ds ,
    \end{align*}
    where the first integral on the right hand side is bounded by a constant (depending on $\varepsilon$).
    
    Plugging this in we get
    \begin{align*}
      I_k(x) &\lesssim 2^{-k\gamma} \Big( \int_0^\infty \int_t^\infty \Big( \frac{t}{s} \Big)^{2(\alpha - \varepsilon)}
      \| 1_{B(x,2^{k+1}s)} \hat{u}(\cdot , s) \|_{L^2}^2 \Ds \, \frac{\textup{d}t}{tV(x,t)} \Big)^{1/2} \\
      &= \Big( \int_0^\infty \| 1_{B(x,2^{k+1}s)} \hat{u}(\cdot , s) \|_{L^2}^2 \int_0^s 
      \Big( \frac{t}{s} \Big)^{2(\alpha - \varepsilon)} \,\frac{\textup{d}t}{tV(x,t)} \Ds , 
    \end{align*}
    where the integration limits are obtained from the identity $1_{(t,\infty)}(s) = 1_{(0,s)}(t)$.
    
    To estimate the inner integral we proceed as follows:
    \begin{align*}
      \int_0^s \Big( \frac{t}{s} \Big)^{2(\alpha - \varepsilon)} \,\frac{\textup{d}t}{tV(x,t)}
      &= \sum_{j=0}^\infty \int_{2^{-(j+1)}s}^{2^{-j}s} 
      \Big( \frac{t}{s} \Big)^{2(\alpha - \varepsilon)} \,\frac{\textup{d}t}{tV(x,t)} \\
      &\leq \sum_{j=0}^\infty \frac{1}{V(x,2^{-(j+1)}s)} \int_{2^{-(j+1)}s}^{2^{-j}s} 
      \Big( \frac{t}{s} \Big)^{2(\alpha - \varepsilon)} \Dt \\
      &\lesssim \sum_{j=0}^\infty \frac{2^{nj}}{V(x,s)} 2^{-j(\alpha - \varepsilon)} \\
      &= \frac{1}{V(x,s)} \sum_{j=0}^\infty 2^{-j(\alpha - \varepsilon -n)} \leq \frac{1}{V(x,s)} ,
    \end{align*}
    where $\varepsilon$ is chosen small enough so that $\alpha - \varepsilon > n$.
    
    We have now established
    \begin{equation*}
      I_k(x) \lesssim 2^{-k\gamma} \Big( \int_0^\infty \int_{B(x,2^{k+1}s)} |\hat{u}(y,s)|^2 \D\mu(y) 
      \, \frac{\textup{d}s}{s V(x,s)} \Big)^{1/2} .
    \end{equation*}
    For $y\in B(x,2^{k+1}s)$ we have
    \begin{equation*}
      \frac{1}{V(x,s)} \leq \Big( 1 + \frac{d(x,y)}{s} \Big)^{n_0} \frac{1}{V(y,s)}
      \lesssim 2^{n_0k} \frac{1}{V(y,s)}
    \end{equation*}
    and so
    \begin{equation*}
      I_k(x) \lesssim 2^{-k(\gamma -n_0/2)} \Big( \iint_{\Gamma_{2^{k+1}}(x)} |\hat{u}(y,s)|^2 \,
      \frac{\textup{d}\mu (y) \,\textup{d}s}{sV(y,s)} \Big)^{1/2} .
    \end{equation*}
    
    In other words we have shown that
    \begin{equation}
    \label{S_infty}
      \Aa (S_\infty u)(x) \leq \sum_{k=0}^\infty \Aa u_k (x) \lesssim \sum_{k=0}^\infty 2^{-k(\gamma -n_0/2)} \Aa_{2^{k+1}}u(x) .
    \end{equation}

    \emph{The operator $S_0$}:
    To estimate $\Aa (S_0 u)(x)$ by a sum of $\Aa_{2^{k+m+2}} u(x)$'s for a fixed $x\in M$   
    we write
    \begin{equation*}
      S_0 u(\cdot , t) = \sum_{k,m=0}^\infty \int_{2^{-(m+1)}t}^{2^{-m}t} K(t,s) ( 1_{C_k(x,t)} u(\cdot , s)) \Ds .
    \end{equation*}
    For a fixed $\xi^*\in X^*$ we again write $\hat{u}(\cdot ,s) = \la u(\cdot ,s) , \xi^* \ra$ and estimate as above:
    \begin{align*}
      I_{k,m}(x) :&= \Big( \iint_{\Gamma (x)} \Big| 
      \int_{2^{-(m+1)}t}^{2^{-m}t} K(t,s) ( 1_{C_k(x,t)} \hat{u}(\cdot , s)) \Ds \Big|^2 
      \,\frac{\textup{d}\mu (y) \,\textup{d}t}{tV(y,t)} \Big)^{1/2} \\
      &\lesssim \Big( \int_0^\infty \Big( \int_{2^{-(m+1)}t}^{2^{-m}t} \| 1_{B(x,t)} K(t,s) ( 1_{C_k(x,t)} \hat{u}(\cdot , s) )
      \|_{L^2} \Ds \Big)^2 \,\frac{\textup{d}t}{tV(x,t)} \Big)^{1/2} .
    \end{align*}
    
    By \eqref{OD}, we have
    \begin{equation*}
      \| 1_{B(x,t)} K(t,s)(1_{C_k(x,t)}\hat{u}(\cdot ,s )) \|_{L^2}
      \lesssim \Big( \frac{s}{t} \Big)^\beta 2^{-k\gamma} \| 1_{B(x,2^{k+1}t)} \hat{u}(\cdot , s) \|_{L^2}
    \end{equation*}
    and so by H{\"o}lder's inequality,
    \begin{align*}
      \Big( \int_{2^{-(m+1)}t}^{2^{-m}t} \| 1_{B(x,t)} K(t,s) ( 1_{C_k(x,t)} \hat{u}(\cdot , s) )
      \|_{L^2} \Ds \Big)^2
      &\lesssim \int_{2^{-(m+1)}t}^{2^{-m}t} \Big( \frac{s}{t} \Big)^{2\beta}
       4^{-k\gamma} \| 1_{B(x,2^{k+1}t)} \hat{u}(\cdot , s) \|_{L^2}^2 \Ds .
    \end{align*}
    
    Plugging this in we obtain
    \begin{align*}
      I_{k,m}(x) &\lesssim 2^{-k\gamma}2^{-m\beta} \Big( \int_0^\infty \int_{2^{-(m+1)}t}^{2^{-m}t}
      \| 1_{B(x,2^{k+1}t)} \hat{u}(\cdot , s) \|_{L^2}^2 \Ds \Big)^{1/2} \\
      &\leq 2^{-k\gamma}2^{-m\beta} \int_0^\infty \int_{B(x,2^{k+m+2}s)} |\hat{u}(y,s)|^2 \mu (y)
      \int_{2^ms}^{2^{m+1}s} \frac{\textup{d}t}{tV(x,t)} \Ds ,
    \end{align*}
    where the exchange of the order of integration is justified by the fact that if
    $2^{-(m+1)}t < s \leq 2^{-m}t$, then $2^ms \leq t < 2^{m+1}s$ and $B(x,2^{k+1}t) \subset B(x,2^{k+m+2}s)$.
    
    When $y\in B(x,2^{k+m+2}s)$ we have
    \begin{align*}
      \int_{2^ms}^{2^{m+1}s} \frac{\textup{d}t}{tV(x,t)} \leq \frac{1}{V(x,2^ms)}
      \lesssim \Big( 1 + \frac{d(x,y)}{2^ms} \Big)^{n_0} \frac{1}{V(y,2^ms)}
      \lesssim \frac{2^{kn_0}}{V(x,s)} 
    \end{align*}
    and so
    \begin{equation*}
      I_{k,m}(x) \lesssim 2^{-k(\gamma -n_0/2)}2^{-m\beta} \Big( \iint_{\Gamma_{2^{k+m+2}}(x)} |\hat{u}(y,s)|^2
      \,\frac{\textup{d}\mu (y)\,\textup{d}s}{sV(y,s)} \Big)^{1/2} .
    \end{equation*}
    
    Again, by Covariance domination, we obtain
    \begin{equation}
    \label{S_0}
      \Aa (S_0 u)(x) \lesssim \sum_{k,m=0}^\infty 2^{-k(\gamma -n_0/2)}2^{-m\beta} \Aa_{2^{k+m+2}}u(x) .
    \end{equation}

    \emph{The operator $S$}: Let $1\leq p < \infty$. We bring together the estimates for $S_\infty$ and $S_0$. 
    From \eqref{S_infty} we obtain using change of aperture (Propositions \ref{refl} and \ref{changeofaperture})
    \begin{equation*}
      \| \Aa (S_\infty u) \|_{L^p} \lesssim \sum_{k=0}^\infty 2^{-k(\gamma -n_0/2)} \| \Aa_{2^{k+1}}u \|_{L^p}
      \lesssim_\varepsilon \sum_{k=0}^\infty 2^{-k(\gamma - 3n/2 - \varepsilon )} \| \Aa u \|_{L^p} .
    \end{equation*}
    Moreover, from \eqref{S_0} we obtain in a similar fashion that
    \begin{equation*}
      \| \Aa (S_0 u) \|_{L^p} 
      \lesssim \sum_{k,m=0}^\infty 2^{-k(\gamma -n_0/2)} 2^{-m\beta}  \| \Aa_{2^{k+m+2}} u \|_{L^p} 
      \lesssim_\varepsilon \sum_{k,m=0}^\infty 2^{-k(\gamma - 3n/2 - \varepsilon)} 2^{-m(\beta - n - \varepsilon)}  \| \Aa u \|_{L^p} .
    \end{equation*}
    Consequently, choosing $\varepsilon$ small enough so that 
    $\gamma - \varepsilon > 3n/2$ and $\beta - \varepsilon > n$ we get
    \begin{equation*}
      \| Su \|_{T^p(X)} \leq \| S_\infty u \|_{T^p(X)} + \| S_0 u \|_{T^p(X)} \lesssim \| u \|_{T^p(X)} .
    \end{equation*}
  \end{proof}
\end{theorem}

\section{Hardy spaces}

We make the following assumptions:

\begin{itemize}
\item Let $(M,d,\mu)$ be a complete doubling metric measure space and assume that it has the cone covering property.
\item
Let $L$ be a non-negative self-adjoint operator on $L^2(M)$ and assume that it generates an analytic semigroup
$(e^{-tL})_{t>0}$, which satisfies the following 
\emph{off-diagonal estimates}: 
There exists a constant $c$ such that
for every $t > 0$ we have
\begin{equation*}
  \| 1_{E'} e^{-tL} (1_Ef) \|_{L^2} \lesssim \exp \Big( -\frac{d(E,E')^2}{ct} \Big) \| 1_Ef \|_{L^2}
\end{equation*}
whenever $E,E'\subset M$ and $f\in L^2(M)$. Sets $E$ and $E'$ in such
a context are assumed, without separate mention, to be measurable.
Denote by $\dom (L)$ and $\ran (L)$ the domain and the range of $L$ on $L^2(M)$.
\item Let $X$ be a UMD space.
\end{itemize}

Recall that on a complete (connected) Riemannian manifold with non-negative sectional curvature the volume measure is doubling with respect
to the geodesic distance. Moreover, the Laplace--Beltrami operator on such a space 
satisfies the off-diagonal estimates, regardless of
curvature. See \cite[Section 1]{AMR} and \cite[Section 3.1]{HOFMANNHARDY} for further discussion and references.

\subsection{Definition and basic properties}
We now define the Hardy spaces and express the conical square function in terms of the tent space norm:
\begin{def*}
  Let $1\leq p < \infty$ and let $N$ be a positive integer. 
  The Hardy space $H^p_{L,N}(X)$ associated with $L$ is defined as the completion of
  $\overline{\ran (L)} \otimes X$ with respect to
  \begin{equation*}
    \| f \|_{H^p_{L,N}(X)} := \| Q_N f \|_{T^p(X)} , \quad \textup{where} \quad Q_Nf(y,t) = (t^2L)^N e^{-t^2L}f(y) ,
    \quad f\in \overline{\ran (L)} \otimes X .
  \end{equation*}
\end{def*}

\begin{remark}
  Note that by the scalar-valued theory (see \cite[Section 4.1]{HOFMANNHARDY}), 
  $Q_Nf \in T^2 \otimes X$ whenever $f\in\overline{\ran (L)} \otimes X$.
\end{remark}

Recall the Calder\'{o}n reproducing formula (the proof of which follows by spectral theory):
For every positive integer $N$ there exists a constant $c$ such that
\begin{equation*}
  f = c \int_0^\infty (t^2L)^{2N} e^{-2t^2L} f \Dt
\end{equation*}
whenever $f\in \overline{\ran (L)} \otimes X$.

We now define, for each positive integer $N$, the mapping
\begin{equation*}
  \pi_N u = \int_0^\infty (t^2L)^N e^{-t^2L}u(\cdot , t) \Dt , \quad u\in T^2 \otimes X ,
\end{equation*}
with which the reproducing formula can be written as $f = c \pi_N Q_Nf$. Here the integral is understood as a limit
in $L^2$ of the integrals $\int_\varepsilon^R$ as $\varepsilon \to 0$ and $R\to\infty$. In what follows, Fubini's theorem
applied to this integral is interpreted by first considering the finite integrals $\int_\varepsilon^R$ and then using
Lebesgue's dominated convergence to pass to the limit.

Note that $Q_N$ and $\pi_N$ are formally adjoint in the sense that for $f\in\overline{\ran (L)} \otimes X$ and 
$v\in T^2 \otimes X^*$ we have
\begin{align*}
  \la Q_Nf , v \ra &= \int_0^\infty \int_M \la (t^2L)^N e^{-t^2L}f(\cdot ) , v(\cdot , t) \ra \D\mu \Dt \\
  &= \int_0^\infty \int_M \la f(\cdot) , (t^2L)^N e^{-t^2L}v(\cdot ,t) \ra \D\mu \Dt \\
  &= \int_M \la f(\cdot ) , \int_0^\infty (t^2L)^N e^{-t^2L}v(\cdot ,t) \Dt \ra \D\mu \\
  &= \la f , \pi_N v \ra . 
\end{align*}

In order to make use of Theorem \ref{intop} in proving, 
for instance, the boundedness of $\pi_N$ from $T^p(X)$ to $H^p_L(X)$ (and the boundedness of the 
$H^\infty$-functional calculus of $L$ on $H^p_L(X)$) we need some off-diagonal estimates of the form \eqref{OD}
for the kernels of our integral operators. There is an abundance of such estimates in the literature and a 
suitable version of Lemma \ref{offdiaglemma} could be obtained directly from sophisticated results like
\cite[Lemma 2.40]{SECONDORDER}. However, taking into account the simplicity of our situation,
we can afford to give some indication of the proof. The first off-diagonal estimate in the following lemma can be found, 
for instance, in \cite[Proposition 3.1]{HOFMANNHARDY}. The second estimate, which is a special case of 
\cite[Lemma 2.28]{SECONDORDER}, contains the heart of the functional calculus in the sense that there and only there
the holomorphicity of $\phi$ is put to use. Note that when $\phi$ is a bounded holomorphic function in a sector
$\{ \zeta\in\C\setminus \{ 0 \} : |\arg \zeta | < \sigma \}$ we can define $\phi (L)f$ by spectral theory for all
$f\in\overline{\ran (L)} \otimes X$.

\begin{lemma}
\label{HMMoffdiag}
  Let $k$ be a non-negative integer and let $\phi$ be a bounded holomorphic function in a sector.
  For all $E,E'\subset M$ and every $f\in L^2(M)$ we have the exponential off-diagonal estimate
  \begin{equation*}
    \| 1_{E'}(t^2L)^k e^{-t^2L} (1_E f ) \|_{L^2} \lesssim \exp \Big( -\frac{d(E,E')^2}{ct^2} \Big) \| 1_E f \|_{L^2} , 
    \quad t > 0 ,
  \end{equation*}
  and the polynomial off-diagonal estimate
  \begin{equation*}
    \| 1_{E'} \phi (L) (t^2L)^k e^{-t^2L} (1_E f ) \|_{L^2} \lesssim \| \phi \|_\infty 
    \Big( 1 + \frac{d(E,E')^2}{t^2} \Big)^{-k} \| 1_E f \|_{L^2} , \quad t > 0.
  \end{equation*}
\end{lemma}

\begin{lemma}
\label{offdiaglemma}
  Let $N,N' \geq 1$ and let $\phi$ be a bounded holomorphic function in a sector. Then
  for all $E,E'\subset M$ and every $f\in L^2(M)$ we have
  \begin{align*}
    \| 1_{E'} (t^2L)^N e^{-t^2L} \phi (L) (s^2L)^{N'} e^{-s^2L} (1_E f) \|_{L^2} \\ 
    \lesssim \| \phi \|_\infty
    &\min \Big( \frac{t^{2N}}{s^{2N}} , \frac{s^{2N'}}{t^{2N'}} \Big)
    \Big( 1 + \frac{d(E,E')}{\max (t,s)} \Big)^{-2(N+N')} \| 1_Ef \|_{L^2}
  \end{align*}
   whenever $t,s > 0$.
  \begin{proof}
    We make use of the fact that off-diagonal estimates (both exponential and polynomial) are stable under
    compositions in the sense of \cite[Lemma 2.22]{SECONDORDER} and \cite[Lemma 6.2]{MORRISREPR}.
    For $t \leq s$ the result follows by writing
    \begin{equation*}
      (t^2L)^N e^{-t^2L} \phi (L) (s^2L)^{N'} e^{-s^2L}
      = \Big( \frac{t}{s} \Big)^{2N} e^{-t^2L} \phi (L) (s^2L)^{N+N'}e^{-s^2L}
    \end{equation*}
    and applying Lemma \ref{HMMoffdiag} separately for $(e^{-t^2L})_{t>0}$ and $(\phi (L)(s^2L)^{N+N'} e^{-s^2L})_{s>0}$.
    Similarly, for $s\leq t$ we write
    \begin{equation*}
      (t^2L)^N e^{-t^2L} \phi (L) (s^2L)^{N'} e^{-s^2L}
      = \Big( \frac{s}{t} \Big)^{2N'} \phi (L) (t^2L)^{N+N'} e^{-t^2L} e^{-s^2L}
    \end{equation*}
    and applying Lemma \ref{HMMoffdiag} for $(\phi (L)(t^2L)^{N+N'} e^{-t^2L})_{t>0}$ and $(e^{-s^2L})_{s>0}$.
  \end{proof}
\end{lemma}

For a real number $\alpha$ we denote by $\lfloor \alpha \rfloor$
the largest integer not greater than $\alpha$.

\begin{prop}
\label{pibounded}
  Let $1\leq p < \infty$. For every $N\geq \lfloor n/2 \rfloor + 1$, 
  $\pi_N$ defines a bounded surjection from $T^p(X)$ onto $H^p_{L,N}(X)$.
  \begin{proof}
    For boundedness it suffices to consider the integral operator
    \begin{equation*}
      Q_N\pi_N u = \int_0^\infty (t^2L)^N e^{-t^2L} (s^2L)^N e^{-s^2L} u(\cdot , s) \Ds ,
    \end{equation*}
    the kernel of which, by Lemma \ref{offdiaglemma}, satisfies 
    the estimate \eqref{OD} with
    $\gamma = 4N > 3n/2$ and $\alpha = \beta = 2N > n$.
    
    Surjectivity follows immediately from the facts that, by definition, 
    $Q_N$ is an isometric embedding (into a complete space),
    and $c\pi_N$ is its continuous left inverse on the dense set $\overline{\ran (L)} \otimes X$.
  \end{proof}  
\end{prop}

The following theorem is a part of our second main result and can be thought of as an extension of Theorem 7.10 in
\cite{HVNPCONICAL} to the endpoint $p=1$:

\begin{theorem}
\label{main2}
  Let $1\leq p < \infty$. 
  Then 
  \begin{itemize}
  \item $H^p_{L,N}(X) = H^p_{L,N'}(X) =: H^p_L(X)$ whenever 
  $N,N'\geq \lfloor n/2 \rfloor + 1$,
  \item $L$ has a bounded $H^\infty$-functional
  calculus of any angle on $H^p_L(X)$, that is, if $\phi$ is a bounded holomorphic function in a sector, then
  \begin{equation*}
    \| \phi (L) f \|_{H^p_L(X)} \lesssim \| \phi \|_\infty \| f \|_{H^p_L(X)}
  \end{equation*}
  for all $f\in \overline{\ran (L)} \otimes X$.
  \end{itemize}
  \begin{proof}
    Assume that $\phi$ is a bounded holomorphic function in a sector.
    We use the reproducing formula to write
    \begin{align*}
      Q_N\phi (L) f(\cdot , t) &= (t^2L)^N e^{-t^2L}\phi (L)f \\
      &= c \int_0^\infty (t^2L)^N e^{-t^2L}\phi (L) (s^2L)^{2N'}e^{-2s^2L}f \Ds \\
      &= \int_0^\infty K(t,s) Q_{N'}f (\cdot , s) \Ds .
    \end{align*}    
    By Lemma \ref{offdiaglemma} the kernel
    \begin{equation*}
      K(t,s) = c (t^2L)^Ne^{-t^2L} \phi (L) (s^2L)^{N'}e^{-s^2L}
    \end{equation*}
    satisfies estimate \eqref{OD} with parameters
    $\gamma > 3n/2$ and $\alpha, \beta > n$ and    
    a constant depending on $\| \phi \|_\infty$.
    
    The first statement follows by considering $\phi$ identically one.
  \end{proof}
\end{theorem}

\begin{prop}
\label{hardyduality}
  Let $1 < p < \infty$. Then $H^p_L(X)^* \simeq H^{p'}_L(X^*)$ and the duality is realized via
  \begin{equation*}
    \la f , g \ra = \int_M \la f(x) , g(x) \ra \D\mu (x) , \quad f\in\overline{\ran (L)} \otimes X , 
    \quad g\in\overline{\ran (L)} \otimes X^* .
  \end{equation*}
  \begin{proof}
    Fix an $N\geq \lfloor n/2 \rfloor + 1$ and abbreviate $Q$ and $\pi$ for $Q_N$ and $\pi_N$. The pairing in the statement arises 
    from the identification of $H^p_L(X)$ as the complemented subspace
    $QH^p_L(X) = Q\pi T^p(X)$ of $T^p(X)$. The projection $Q\pi$ on $T^p(X)$ has the adjoint
    $(Q\pi)^* = \pi^* Q^* = Q\pi$ on $T^p(X)^* \simeq T^{p'}(X^*)$ and therefore
    \begin{equation*}
      H^p_L(X)^* \simeq (Q\pi T^p(X))^* \simeq Q\pi T^{p'}(X^*) \simeq H^{p'}_L(X^*).
    \end{equation*}
  \end{proof}
\end{prop}

\begin{remark}
  From Theorem \ref{endpointduality} it follows that bounded linear functionals on $H^1_L(X)$ are of the form
  $f\mapsto \la Qf , v \ra$, where $v\in T^\infty (X^*)$. We will not attempt to describe $H^1_L(X)^*$ as a space of
  functions on $M$.
\end{remark}

The other part of our second main result extends the complex interpolation scale of vector-valued 
Hardy spaces to the endpoint $p=1$ (cf. Corollary 7.2 in \cite{HVNPCONICAL}):

\begin{theorem}
\label{main1}
  Let $1\leq p_0 \leq p_1 < \infty$. Then
  \begin{equation*}
    [QH^{p_0}_L (X),QH^{p_1}_L(X)]_\theta = QH^p_L(X) , \quad \text{where}\quad 
    \frac{1}{p} = \frac{1-\theta}{p_0} + \frac{\theta}{p_1} .
  \end{equation*}
  \begin{proof}
    This follows from interpolation of tent spaces (Corollary \ref{tentmain1}) 
    along with boundedness of the projection $Q\pi$ (Proposition \ref{pibounded} and the proof of Proposition \ref{hardyduality}) by means of interpolation of complemented subspaces
    (see \cite[Corollary 7.2]{HVNPCONICAL} and the references therein). 
  \end{proof}
\end{theorem}


\subsection{Atoms}

In order to transfer the atomic decomposition from $T^1(X)$ to $H^1_L(X)$ we proceed as in 
\cite[Subsection 4.3]{HOFMANNHARDY}. Relying on the self-adjointness of $L$ we may define, as in
\cite[Lemmas 3.5 and 4.11]{HOFMANNHARDY}\footnote{More precisely, we put $\Phi_t = \widehat{\phi}(t\sqrt{L})$, where $\phi$ is smooth and compactly supported around $0$ in $\R$. The desired properties are expressed in equations (4.21) and (3.12) in \cite{HOFMANNHARDY}.},
a family $(\Phi_t)_{t>0}$ uniformly bounded operators on $L^2(M)$ such that
\begin{itemize}
\item
for all positive integers $N, N'$ there exists a constant $c$ such that
\begin{equation*}
  f = c \int_0^\infty (t^2L)^{N+N'} \Phi_t e^{-t^2L}f \Dt , \quad f\in \overline{\ran (L)} \otimes X ,
\end{equation*}
\item
for all non-negative integers $k$ the family $((t^2L)^k \Phi_t)_{t>0}$ of bounded operators on $L^2(M)$
has finite speed of propagation in the sense that if $t\leq d(E,E')$
for some $E,E'\subset M$, then
$1_{E'}(t^2L)^k \Phi_t (1_Ef) = 0$ whenever $f\in L^2(M)$.
\end{itemize}

We now define the operators
\begin{equation*}
  \widetilde{Q}_N f (y,t) = (t^2L)^N \Phi_t f(y) , \quad f\in \overline{\ran (L)} \otimes X ,
\end{equation*}
and
\begin{equation*}
  \widetilde{\pi}_{N'}u = \int_0^\infty (t^2L)^{N'} \Phi_t u(\cdot , t) \Dt , \quad u\in T^2 \otimes X ,
\end{equation*}
with which the new reproducing formula can be written as 
$f = c \pi_{N'} \widetilde{Q}_Nf = c\widetilde{\pi}_{N'}Q_Nf$.

\begin{prop}
\label{pivarphi}
  Let $1\leq p < \infty $. The operators $\widetilde{Q}_N : H^p_L(X) \to T^p(X)$ and 
  $\widetilde{\pi}_N : T^p(X) \to H^p_L(X)$ are bounded whenever 
  $N\geq \lfloor n/2 \rfloor + 1$.
  \begin{proof}
    Again, it suffices to view $\widetilde{Q}_N$ and $\widetilde{\pi}_N$ as integral operators.
    Indeed,
    \begin{equation*}
      \widetilde{Q}_N f (\cdot , t) = \widetilde{Q}_N \pi_N Q_N f (\cdot , t)
      = c \int_0^\infty (t^2L)^N \Phi_t (s^2L)^N e^{-s^2L} Q_Nf(\cdot , s) \Ds
    \end{equation*}
    and
    \begin{equation*}
      Q_N\widetilde{\pi}_Nu = c \int_0^\infty (t^2L)^Ne^{-t^2L} (s^2L)^N \Phi_s u(\cdot , s) \Ds .
    \end{equation*}
    To see that the kernels of these integral operators satisfy \eqref{OD} one argues as in Lemma
    \ref{offdiaglemma} with $(t^2L)^N \Phi_t$ replacing $(t^2L)^N e^{-t^2L}$. Note that
    the exponential off-diagonal estimates are immediate from the fact that
    $1_{E'}(t^2L)^k \Phi_t(1_Ef) = 0$ when $t\leq d(E,E')$.
  \end{proof}
\end{prop}

\begin{def*}
  A function $m\in L^2(M) \otimes X$ is said to be an $L$-\emph{atom} of order $K$
  associated with a ball $B\subset M$ if there exists a function $\widetilde{m} \in \dom (L^K) \otimes X$, 
  such that
  \begin{itemize}
    \item $m = L^K \widetilde{m}$,
    \item $\supp m \subset B$,
    \item $\| (r_B^2L)^k \widetilde{m} \|_{H_L^2(X)} \leq r_B^{2K}\mu (B)^{-1/2}$, $k=0,1,\ldots , K$.
  \end{itemize}
\end{def*}

\begin{remark}
  It is not clear if all $L$-atoms belong to $H^1_L(X)$ as in the scalar-valued setting (see
  \cite[Proposition 4.4]{HOFMANNHARDY}).
\end{remark}

\begin{prop}
\label{atomstomolecules}
  Let $a\in T^1 \otimes X$ be an atom in $T(B)$ for a ball $B\subset M$ and let $K$ be a positive integer. Then 
  $\widetilde{\pi}_{N+K} a \in H_L^1(X)$ is an (constant multiple of an) $L$-atom of order $K$ 
  in $2B$ whenever $N\geq \lfloor n/2 \rfloor + 1$. 
  \begin{proof}
    Choosing
    \begin{equation*}
      \widetilde{m} = \int_0^{r_B} t^{2(N+K)} L^N \Phi_t a(\cdot , t) \Dt \in D(L^K) \otimes X
    \end{equation*}
    we obtain
    \begin{equation*}
      L^K\widetilde{m} = \int_0^{r_B} (t^2L)^{N+K} \Phi_t a(\cdot , t) \Dt = \widetilde{\pi}_{N+K}a ,
    \end{equation*}
    as usual (cf. \cite[Lemma 4.11]{HOFMANNHARDY}).
    
    To see that $\supp \widetilde{\pi}_{N+K}a \subset 2B$ it suffices to note that for all $t \leq r_B$
    we have $\supp a(\cdot , t) \subset B$ and thus also
    \begin{equation*}
      1_{M\setminus 2B}(t^2L)^{N+K}\Phi_t a(\cdot , t) = 0 .
    \end{equation*}
    
    For the size condition we pair $(r_B^2L)^k\widetilde{m}$ with an arbitrary $g\in \overline{\ran (L)}\otimes X^*$
    and estimate as follows:
    \begin{align*}
      \Big| \int_M \la (r_B^2L)^k\widetilde{m}(\cdot) , g(\cdot) \ra \D\mu \Big|
      &= \Big| \int_M \Big\la \int_0^{r_B} t^{2(N+K)}r_B^{2k}L^{N+k} \Phi_t a(\cdot , t) \Dt ,
      g(\cdot) \Big\ra \D\mu \Big| \\
      &= \Big| \int_0^{r_B} t^{2(N+K)} r_B^{2k} \int_M \la a(\cdot ,t),L^{N+k} \Phi_t g(\cdot) \ra \D\mu \Dt \Big| \\
      &\leq r_B^{2K} \iint_{M^+} |\la a(\cdot ,t) , (t^2L)^{N+k} \Phi_t g(\cdot ) \ra | \, \frac{\textup{d}\mu \D t}{t} \\
      &\lesssim r_B^{2K} \| a \|_{T^2(X)} \| \widetilde{Q}_{N+k}g \|_{T^2(X^*)} \\
      &\lesssim r_B^{2K} \mu (B)^{-1/2} \| g \|_{H_L^2(X^*)} .
    \end{align*}
    The required norm estimate follows then by duality (Proposition \ref{hardyduality}).
  \end{proof}
\end{prop}

\begin{theorem}
\label{Hardyatomicdec}
  Every $f\in \overline{\ran (L)} \otimes X$ in $H^1_L(X)$ can be written, for any positive integer $K$, 
  as a sum of $L$-atoms $m_k\in H^1_L(X)$ of order $K$ so that
  \begin{equation*}
    f = \sum_k \lambda_k m_k ,
  \end{equation*}
  where the sum converges in both $H^1_L(X)$ and $L^2(X)$, 
  and the scalars $\lambda_k$ satisfy
  \begin{equation*}
    \sum_k |\lambda_k| \eqsim \| f \|_{H^1_L(X)} .
  \end{equation*}
  Moreover, if $H^2_L(X) = L^2(X)$, then the sum converges also in
  $L^1(X)$.
  \begin{proof}
    Let $K$ be a positive integer.
    Given an $f\in \overline{\ran (L)} \otimes X$ in $H^1_L(X)$ we fix an $N\geq \lfloor n/2 \rfloor + 1$ and decompose $Q_Nf\in T^1 \otimes X$ 
    into atoms $a_k$ by Theorem \ref{atomicdec} so that
    \begin{equation*}
      Q_Nf = \sum_k \lambda_k a_k \quad \textup{and} \quad \sum_k |\lambda_k| \eqsim \| Q_N f \|_{T^1(X)} 
      \eqsim \| f \|_{H^1_L(X)} ,
    \end{equation*}
    where the sum for $Q_Nf$ converges in both $T^1(X)$ and $T^2(X)$.
    Consequently, for a constant $c$ we have
    \begin{equation}
    \label{atomsum}
      f = c\widetilde{\pi}_{N+K}Q_Nf = c \sum_k \lambda_k \widetilde{\pi}_{N+K}a_k ,
    \end{equation}    
    where, by Proposition \ref{atomstomolecules}, 
    $\widetilde{\pi}_{N+K}a_k$ are (constant multiples of) $L$-atoms of order $K$ and the sum converges in both $H^1_L(X)$ and $H^2_L(X)$.
    
    Assuming that $H^2_L(X) = L^2(X)$, we see that $L$-atoms are
    uniformly bounded in $L^1(X)$. Indeed,
    an $L$-atom $m\in L^2(M)\otimes X$ associated with a 
    ball $B$ satisfies
    \begin{equation*}
      \| m \|_{L^1(X)} \leq \mu (B)^{1/2} \| m \|_{L^2(X)} \lesssim \mu (B)^{1/2} \| m \|_{H^2_L(X)} \leq 1 .
    \end{equation*}
    The right-hand side of \eqref{atomsum} is therefore absolutely
    summable in $L^1(X)$ and converges in $L^1(X)$ to a limit which must
    coincide with its limit in $L^2(X)$, that is, $f$.
  \end{proof}
\end{theorem}

\begin{cor}
  Suppose that $H_L^2(X) = L^2(X)$. For every $f\in L^2(M) \otimes X$ we have
  \begin{itemize}
    \item $\| f \|_{L^p(X)} \lesssim \| f \|_{H^p_L(X)}$ when $1\leq p \leq 2$,
    \item $\| f \|_{H^p_L(X)} \lesssim \| f \|_{L^p(X)}$ when $2\leq p < \infty$.
  \end{itemize}
  \begin{proof}
    Every $f\in L^2(M) \otimes X$ in $H^1_L(X)$ admits, by 
    Theorem \ref{Hardyatomicdec}, an $L^1(X)$-convergent 
    decomposition into $L$-atoms 
    (which are uniformly bounded in $L^1(X)$) and so
    \begin{equation*}
      \| f \|_{L^1(X)} \leq \sum_k |\lambda_k| \eqsim \| f \|_{H^1_L(X)} .
    \end{equation*}
    By interpolation (Theorem \ref{main1}) we have $\| f \|_{L^p(X)} \lesssim \| f \|_{H^p_L(X)}$ when $1\leq p \leq 2$.
    
    The second inequality $\| f \|_{H^p_L(X)} \lesssim \| f \|_{L^p(X)}$ for $2\leq p < \infty$ follows from the first 
    by duality:
    \begin{align*}
      \| f \|_{H^p_L(X)} &\eqsim \sup \{ | \la f,g \ra | : g\in L^2(M) \otimes X^*, \| g \|_{H^{p'}_L(X^*)} \leq 1 \} \\
      &\lesssim \sup \{ | \la f,g \ra | : g\in L^2(M) \otimes X^*, \| g \|_{L^{p'}(X^*)} \leq 1 \} \eqsim \| f \|_{L^p(X)} .
    \end{align*}
  \end{proof}
\end{cor}

\begin{remark}
  We refrain from addressing the question 
  whether $H^p_L(X)$ embeds in $L^p(X)$ for $1 \leq p \leq 2$ (or vice
  versa for $2\leq p < \infty$). This subtle matter has been discussed
  at length in \cite{MORRISREPR}.
\end{remark}

\begin{example}
\label{example}
  Let $L=\Delta$ be the (non-negative) Laplacian on $M=\R^n$ with the Lebesgue measure. For functions
  $f\in L^2(\R^n) \otimes X$ we have
  \begin{equation*}
    Q_Nf(y,t) = (t^2\Delta)^N e^{-t^2\Delta}f(y) = \int_{\R^n} \Psi_t(y-z)f(z)\D z ,
  \end{equation*} 
  where the Fourier transform of the Schwartz function $\Psi_t$ is given by
  \begin{equation*}
    \widehat{\Psi_t}(\xi) = (t^2|\xi |^2)^N e^{-t^2|\xi |^2} , \quad \xi\in\R^n . 
  \end{equation*}
  As in the proofs of \cite[Theorems 8.2 and 4.8]{HVNPCONICAL} this gives rise to a singular integral operator  
  \begin{equation*}
    Tf(x) = \int_{\R^n} K(x,z) f(z) \D z
  \end{equation*}
  with an operator-valued kernel $K(x,z) \in \Ll (X,\gamma (L^2(\R^{n+1}_+),X))$ so that
  \begin{equation*}
    \| f \|_{H_{\Delta}^p(X)} \eqsim \| Tf \|_{L^p(\gamma (L^2(\R^{n+1}_+),X))}
  \end{equation*}
  for test functions $f\in C^\infty_c(\R^n) \otimes X$.
  
  In the proof of \cite[Theorem 4.8]{HVNPCONICAL} $T$ is shown to be a Calderón--Zygmund operator and thus for
  $1 < p < \infty$ we have
  \begin{equation*}
    \| f \|_{H_{\Delta}^p(X)} \lesssim \| f \|_{L^p(X)} .
  \end{equation*}
  Moreover, the same inequality holds for $X^*$, namely
  \begin{equation*}
    \| g \|_{H_{\Delta}^p(X^*)} \lesssim \| g \|_{L^p(X^*)} ,
  \end{equation*}
  and therefore $H^p_{\Delta}(X) = L^p(X)$ when $1 < p < \infty$.
  
  Let us also remark that $H^1_{\Delta}(X)$ coincides with the \emph{atomic Hardy space} \label{hardydef}
  $H^1_{at}(X)$ which is defined to consist of functions $f\in L^1(X)$
  that can be expressed as sums of (classical) atoms $m_k$ so that
  \begin{equation*}
    f = \sum_k \lambda_k m_k  \quad \textup{and} \quad \| f \|_{H^1_{at}(X)} = \inf \sum_k |\lambda_k | < \infty .
  \end{equation*}
  Here a classical atom is a function $m\in L^2(X)$ which is supported in a ball $B\subset\R^n$ and satisfies
  \begin{equation*}
    \int_B m(x) \D x = 0 \quad \textup{and} \quad \| m \|_{L^2(X)} \leq |B|^{-1/2} .
  \end{equation*}
  Indeed, as a Calderón--Zygmund operator, $T$ is bounded from $H^1_{at}(X)$ to $L^1(\gamma (L^2(\R^{n+1}_+),X))$,
  and thus for all $f\in C^\infty_c(\R^n) \otimes X$ with zero mean we have
  \begin{equation*}
    \| f \|_{H^1_{\Delta}(X)} \lesssim \| f \|_{H^1_{at}(X)} .
  \end{equation*}
  On the other hand, every $L$-atom $m$ is (a constant multiple of) a classical atom since
  \begin{equation*}
    \int_{\R^n} m(x) \D x = \int_{\R^n} \Delta \widetilde{m}(x) \D x = 0 \quad \textup{and} \quad
    \| m \|_{L^2(X)} \lesssim \| m \|_{H^2_{\Delta}(X)} \leq |B|^{-1/2} .
  \end{equation*}
  Theorem \ref{Hardyatomicdec} then guarantees that every
  $f\in L^2(\R^n) \otimes X$ in $H^1_{\Delta}(X)$ satisfies
  \begin{equation*}
    \| f \|_{H^1_{at}(X)} \lesssim \| f \|_{H^1_{\Delta}(X)} .
  \end{equation*} 

\end{example}

\begin{remark}
  For a wide class of Schr{\"o}dinger operators $L=\Delta + V$ with non-negative potentials $V$ on $\R^n$
  (including the harmonic oscillator with $V(x) = |x|^2$) it has been shown by Betancor et al.
  \cite{BETANCORNEW} that the conical square function estimate
  \begin{equation*}
    \| Q_P f \|_{T^p(X)} \eqsim \| f \|_{L^p(X)} , \quad Q_Pf(y,s) = s\sqrt{L} e^{-s\sqrt{L}}f(y),
  \end{equation*}
  associated with the Poisson semigroup, holds for $1 < p < \infty$ 
  whenever $X$ is a UMD space. Such operators $L$ satisfy the off-diagonal
  estimates (see \cite[Chapter 8]{HOFMANNHARDY}) and are therefore within the framework of this article.
  That $\| f \|_{H_L^p(X)} \lesssim \| Q_P f \|_{T^p(X)}$ follows again by means of integral operators on tent spaces
  (cf. the proof of Theorem \ref{main2}). Indeed, the reproducing formula
  \begin{equation*}
    f = c \int_0^\infty (s\sqrt{L})^{2N+1} e^{-2s\sqrt{L}}f \Ds
  \end{equation*}
  is valid (by spectral theory) and the kernel
  \begin{equation*}
    K(t,s) = (t^2L)^Ne^{-t^2L} (s\sqrt{L})^{2N} e^{-s\sqrt{L}}
  \end{equation*}
  satisfies the required estimate \eqref{OD} when 
  $N\geq \lfloor n/2 \rfloor + 1$, which can be seen with the aid of 
  \cite[Lemma 4.15]{HOFMANNHARDY}. As in the example above, we can then argue by duality to see that
  $H_L^p(X)=L^p(X)$ for $1<p<\infty$.
\end{remark}

\begin{appendix}

\section{Completeness and dense subspaces of tent spaces}

\begin{prop}
\label{tentcompleteness}
  For every $1\leq p < \infty$ and $\alpha \geq 1$, the tent space $T^p_\alpha(X)$ 
  is complete and contains $L^2_c(M^+) \otimes X$ as a dense subspace.
\end{prop}

We follow the classical proof of the corresponding scalar-valued result (see \cite[Section 1]{CMSTENTSPACES} and 
\cite[Lemma 3.3 and Proposition 3.4]{AMENTA}).
For simplicity we omit the $\alpha$ as it is immaterial for the proofs and
abbreviate $\| \cdot \|_\gamma$ for $\| \cdot \|_{\gamma (L^2(M^+),X)}$.

\begin{lemma}
\label{compactsubs}
  Let $1\leq p < \infty$ and $u\in T^p(X)$. Then
  \begin{enumerate}
    \item $\| u \|_{T^p(X)} = \sup_K \| u1_K \|_{T^p(X)}$, where the supremum is over compact $K\subset M^+$,
    \item $\inf_K \| u1_{K^c} \|_{T^p(X)} = 0$, where the infimum is over compact $K\subset M^+$,
    \item for every compact $K\subset M^+$ there exists a constant $c_K$ such that 
    \begin{equation*}
      c_K^{-1} \| u1_K \|_{T^p(X)} \leq \| u1_K \|_\gamma \leq c_K \| u \|_{T^p(X)} .
    \end{equation*}
  \end{enumerate}
  \begin{proof}
    For the first claim, write 
    $\Gamma (x;\varepsilon ) = \{ (y,t)\in\Gamma (x) : \varepsilon < t < 1/\varepsilon \}$ and note that
    as $\varepsilon$ tends to zero, the increasing sequence $\| u1_{\Gamma (x;\varepsilon)} \|_\gamma$
    tends to $\| u1_{\Gamma (x)} \|_\gamma$. Therefore,
    \begin{align*}
      \| u \|_{T^p(X)} &= \lim_{\varepsilon\to 0} \Big( \int_M \| u1_{\Gamma (x;\varepsilon)} \|_\gamma^p \D\mu (x)
      \Big)^{1/p} \\
      &= \sup_{\varepsilon , B} \Big( \int_B \| u1_{\Gamma (x;\varepsilon)} \|_\gamma^p \D\mu (x) \Big)^{1/p} \\
      &\leq \sup_K \| u1_K \|_{T^p (X)} ,
    \end{align*}
    because whenever $x$ is in a ball $B\subset M$ and $\varepsilon > 0$, the cone $\Gamma (x;\varepsilon)$ is
    contained in a compact $K\subset M^+$.
    
    The second claim follows by monotone convergence after choosing 
    an increasing (and exhausting) sequence of compact subsets $K$ so that
    for every $x\in M$ the decreasing sequence $\Aa (u1_{K^c})(x) = \| u1_{K^c\cap\Gamma (x)} \|_\gamma$ tends to zero.
    
    To prove the right hand side in the inequality of the third claim, write
    $S(K) = \{ x\in M : \Gamma (x) \cap K \neq \emptyset \}$ and observe that
    $\Aa (u1_K)(x) \leq \| u1_K \|_\gamma$ to obtain
    \begin{equation*}
      \| u1_K \|_{T^p(X)} = \Big( \int_{S(K)} \Aa (u1_K)(x)^p \D\mu (x) \Big)^{1/p}
      \leq \mu (S(K))^{1/p} \| u1_K \|_\gamma . 
    \end{equation*}
    
    The left hand side in the inequality of the third claim follows by choosing a finite number $N(K)$ of 
    (small) balls $B$ so that
    $K\subset \bigcup_B (B\times (0,\infty)) =: \bigcup_B B^+$ and so that for every $x\in B$ we have 
    $K\cap B^+ \subset \Gamma_\alpha (x)$. Then for each $B$ we have $\| u1_{K\cap B^+} \|_\gamma \leq \Aa u(x)$
    when $x\in B$ and therefore
    \begin{equation*}
      \| u1_K \|_\gamma \leq \sum_B \| u1_{K\cap B^+} \|_\gamma 
      \leq \sum_B \Big( \fint_B \Aa u(x)^p \D\mu (x) \Big)^{1/p}
      \leq \frac{N(K)}{\inf_B \mu(B)^{1/p}} \| \Aa u \|_{L^p} = c_K \| u \|_{T^p(X)}
    \end{equation*}
  \end{proof}
\end{lemma}

  \begin{proof}[Proof of Proposition \ref{tentcompleteness}]
    Let $(u_k)$ be a Cauchy sequence in $T^p(X)$. For each compact $K\subset M^+$ we now see by (3) of Lemma 
    \ref{compactsubs} that $(u_k1_K)$ is a Cauchy sequence in $\gamma (L^2(M^+),X)$ and therefore converges to
    a $u^K$. Setting $u=u^K$ on each $L^2(K)$ results in a well-defined linear operator from $L^2_c(M^+)$ to $X$.
    
    To see that $u$ is in $T^p(X)$, fix a compact $K\subset M^+$ and observe that for each $k$,
    \begin{equation*}
      \| u1_K \|_{T^p(X)} \leq \| (u-u_k)1_K \|_{T^p(X)} + \| u_k1_K \|_{T^p(X)}
      \leq c_K \| (u-u_k)1_K \|_\gamma + \| u_k \|_{T^p(X)} .
    \end{equation*}
    Choosing $k$ large enough, we see that $\| u1_K \|_{T^p(X)} \lesssim 1$ independently of $K$, which means that
    $u\in T^p(X)$.
    
    In order to show that $u_k$ converges to $u$ in $T^p(X)$, let $\varepsilon > 0$.
    Choose then a number $N$ so that $\| u_k - u_N \|_{T^p(X)} < \varepsilon$ for all $k\geq N$ and, by (2) of 
    Lemma \ref{compactsubs}, a
    compact $K$ so that $\| (u-u_N)1_{K^c} \|_{T^p(X)} < \varepsilon$. Then for all $k\geq N$,
    \begin{align*}
      \| u - u_k \|_{T^p(X)} &\leq \| (u-u_k)1_K \|_{T^p(X)} + \| (u-u_N)1_{K^c} \|_{T^p(X)} + \| (u_k-u_N)1_{K^c} \|_{T^p(X)} \\
      &\leq c_K \| (u-u_k)1_K \|_\gamma + 2\varepsilon ,
    \end{align*}
    where the first term on the right tends to zero as $k\to\infty$.
    
    Finally, the density of $L^2_c(M^+) \otimes X$ in $T^p(X)$ follows 
    by approximating $u$ by $u1_K$ in $T^p(X)$ and then
    $u1_K$ by a finite rank operator $u'1_K$ in $\gamma (L^2(M^+),X)$.
  \end{proof}

\end{appendix}

\bibliographystyle{plain}

\begin{thebibliography}{10}

\bibitem{AMENTA}
A.~Amenta.
\newblock Tent spaces over metric measure spaces under doubling and related
  assumptions.
\newblock In {\em Operator Theory in Harmonic and Non-commutative Analysis},
  volume 240 of {\em Operator Theory: Advances and Applications}, pages 1--29.
  Springer International Publishing, 2014.

\bibitem{NONNEGATIVE}
A.~Amenta and M.~Kemppainen.
\newblock Non-uniformly local tent spaces.
\newblock {\em Publ. Mat.}, 59(1):245--270, 2015.

\bibitem{MORRISREPR}
P.~Auscher, A.~McIntosh, and A.~J. Morris.
\newblock Calderón reproducing formulas and applications to {H}ardy spaces.
\newblock {\em Rev. Mat. Iberoam.}, (to appear).

\bibitem{AMR}
P.~Auscher, A.~McIntosh, and E.~Russ.
\newblock Hardy spaces of differential forms on {R}iemannian manifolds.
\newblock {\em J. Geom. Anal.}, 18(1):192--248, 2008.

\bibitem{BERGH}
J.~Bergh and J.~L{\"o}fstr{\"o}m.
\newblock {\em Interpolation spaces. {A}n introduction}.
\newblock Springer-Verlag, Berlin, 1976.
\newblock Grundlehren der Mathematischen Wissenschaften, No. 223.

\bibitem{BERNAL}
A.~Bernal.
\newblock Some results on complex interpolation of {$T^p_q$} spaces.
\newblock In {\em Interpolation spaces and related topics ({H}aifa, 1990)},
  volume~5 of {\em Israel Math. Conf. Proc.}, pages 1--10. Bar-Ilan Univ.,
  Ramat Gan, 1992.

\bibitem{BETANCORNEW}
J.~J. Betancor, A.~J. Castro, J.~C. Fari{\~n}a, and L.~Rodr{\'i}guez-Mesa.
\newblock Conical square functions associated with {B}essel, {L}aguerre and
  {S}chr{\"o}dinger operators in {UMD} {B}anach spaces.
\newblock arXiv:1404.5779.

\bibitem{BOURGAINDUALITY}
J.~Bourgain.
\newblock Vector-valued singular integrals and the {$H^1$}-{BMO} duality.
\newblock In {\em Probability theory and harmonic analysis ({C}leveland,
  {O}hio, 1983)}, volume~98 of {\em Monogr. Textbooks Pure Appl. Math.}, pages
  1--19. Dekker, New York, 1986.

\bibitem{BURKHOLDERSINGULAR}
D.~L. Burkholder.
\newblock Martingales and singular integrals in {B}anach spaces.
\newblock In {\em Handbook of the geometry of {B}anach spaces, {V}ol. {I}},
  pages 233--269. North-Holland, Amsterdam, 2001.

\bibitem{CMSTENTSPACES}
R.~R. Coifman, Y.~Meyer, and E.~M. Stein.
\newblock Some new function spaces and their applications to harmonic analysis.
\newblock {\em J. Funct. Anal.}, 62(2):304--335, 1985.

\bibitem{YAGI}
M.~Cowling, I.~Doust, A.~McIntosh, and A.~Yagi.
\newblock Banach space operators with a bounded {$H^\infty$} functional
  calculus.
\newblock {\em J. Austral. Math. Soc. Ser. A}, 60(1):51--89, 1996.

\bibitem{GUERRE}
S.~Guerre-Delabri{\`e}re.
\newblock Some remarks on complex powers of {$(-\Delta)$} and {UMD} spaces.
\newblock {\em Illinois J. Math.}, 35(3):401--407, 1991.

\bibitem{HARBOURE}
E.~Harboure, J.~L. Torrea, and B.~E. Viviani.
\newblock A vector-valued approach to tent spaces.
\newblock {\em J. Analyse Math.}, 56:125--140, 1991.

\bibitem{HIEBER}
M.~Hieber and J.~Pr{\"u}ss.
\newblock Functional calculi for linear operators in vector-valued
  {$L^p$}-spaces via the transference principle.
\newblock {\em Adv. Differential Equations}, 3(6):847--872, 1998.

\bibitem{HOFMANNHARDY}
S.~Hofmann, G.~Lu, D.~Mitrea, M.~Mitrea, and L.~Yan.
\newblock Hardy spaces associated to non-negative self-adjoint operators
  satisfying {D}avies-{G}affney estimates.
\newblock {\em Mem. Amer. Math. Soc.}, 214(1007):vi+78, 2011.

\bibitem{HOFMANNDIV}
S.~Hofmann and S.~Mayboroda.
\newblock Hardy and {BMO} spaces associated to divergence form elliptic
  operators.
\newblock {\em Math. Ann.}, 344(1):37--116, 2009.

\bibitem{SECONDORDER}
S.~Hofmann, S.~Mayboroda, and A.~McIntosh.
\newblock Second order elliptic operators with complex bounded measurable
  coefficients in {$L^p$}, {S}obolev and {H}ardy spaces.
\newblock {\em Ann. Sci. \'Ec. Norm. Sup\'er. (4)}, 44(5):723--800, 2011.

\bibitem{HYTONENEMB}
T.~Hyt{\"o}nen.
\newblock Fourier embeddings and {M}ihlin-type multiplier theorems.
\newblock {\em Math. Nachr.}, 274/275:74--103, 2004.

\bibitem{HYTONENKAIREMA}
T.~Hyt{\"o}nen and A.~Kairema.
\newblock Systems of dyadic cubes in a doubling metric space.
\newblock {\em Colloq. Math.}, 126(1):1--33, 2012.

\bibitem{HVNPCONICAL}
T.~Hyt{\"o}nen, J.~M. A. M.~van Neerven, and P.~Portal.
\newblock Conical square function estimates in {UMD} {B}anach spaces and
  applications to {$H^\infty$}-functional calculi.
\newblock {\em J. Anal. Math.}, 106:317--351, 2008.

\bibitem{HYTONENWEISPARAPRODUCTS}
T.~Hyt{\"o}nen and L.~Weis.
\newblock The {B}anach space-valued {BMO}, {C}arleson's condition, and
  paraproducts.
\newblock {\em J. Fourier Anal. Appl.}, 16(4):495--513, 2010.

\bibitem{TENTSPACES}
M.~Kemppainen.
\newblock The vector-valued tent spaces ${T}^1$ and ${T}^\infty$.
\newblock {\em J. Aust. Math. Soc.}, 97(1):107--126, 2014.

\bibitem{WEISBOOK}
P.~C. Kunstmann and L.~Weis.
\newblock Maximal {$L_p$}-regularity for parabolic equations, {F}ourier
  multiplier theorems and {$H^\infty$}-functional calculus.
\newblock In {\em Functional analytic methods for evolution equations}, volume
  1855 of {\em Lecture Notes in Math.}, pages 65--311. Springer, Berlin, 2004.

\bibitem{MCINTOSH}
A.~McIntosh.
\newblock Operators which have an {$H_\infty$} functional calculus.
\newblock In {\em Miniconference on operator theory and partial differential
  equations ({N}orth {R}yde, 1986)}, volume~14 of {\em Proc. Centre Math. Anal.
  Austral. Nat. Univ.}, pages 210--231. Austral. Nat. Univ., Canberra, 1986.

\bibitem{GAMMARAD}
J.~M. A. M.~van Neerven.
\newblock {$\gamma$}-radonifying operators---a survey.
\newblock In {\em The {AMSI}-{ANU} {W}orkshop on {S}pectral {T}heory and
  {H}armonic {A}nalysis}, volume~44 of {\em Proc. Centre Math. Appl. Austral.
  Nat. Univ.}, pages 1--61. Austral. Nat. Univ., Canberra, 2010.

\bibitem{JVNSTOCHINT}
J.~M. A. M.~van Neerven and L.~Weis.
\newblock Stochastic integration of functions with values in a {B}anach space.
\newblock {\em Studia Math.}, 166(2):131--170, 2005.

\bibitem{RUSS}
E.~Russ.
\newblock The atomic decomposition for tent spaces on spaces of homogeneous
  type.
\newblock In {\em C{MA}/{AMSI} {R}esearch {S}ymposium ``{A}symptotic
  {G}eometric {A}nalysis, {H}armonic {A}nalysis, and {R}elated {T}opics''},
  volume~42 of {\em Proc. Centre Math. Appl. Austral. Nat. Univ.}, pages
  125--135. Austral. Nat. Univ., Canberra, 2007.

\bibitem{ZIMMER}
F.~Zimmermann.
\newblock On vector-valued {F}ourier multiplier theorems.
\newblock {\em Studia Math.}, 93(3):201--222, 1989.

\end{thebibliography}
\def\cprime{$'$} \def\cprime{$'$}

\end{document}